\documentclass[a4paper, 11pt]{amsart}

\usepackage{epsfig}
\usepackage{amsthm,amsfonts}
\usepackage{amssymb,graphicx,color}
\usepackage[all]{xy}
\usepackage{verbatim}
\usepackage{hyperref}
\usepackage{enumitem}

\newtheorem{theorem}{Theorem}[section]
\newtheorem*{theorem*}{Theorem}
\newtheorem{lemma}[theorem]{Lemma}
\newtheorem{corollary}[theorem]{Corollary}
\newtheorem{proposition}[theorem]{Proposition}
\newtheorem{definition}[theorem]{Definition}

\newtheorem{example}[theorem]{Example}
\newtheorem{remark}[theorem]{Remark}
\newtheorem{claim}{Claim}[theorem]
\newtheorem{question}{Question}
\newtheorem*{problem_b}{Bernstein's problem}

\newcommand{\R}{\mathbb{R}}
\newcommand{\C}{\mathbb{C}}
\newcommand{\N}{\mathbb{N}}

\newtheorem{innercustomthm}{Theorem}
\newenvironment{customthm}[1]
  {\renewcommand\theinnercustomthm{#1}\innercustomthm}
  {\endinnercustomthm}
  
\newtheorem{innercustomcor}{Corollary}
\newenvironment{customcor}[1]
  {\renewcommand\theinnercustomcor{#1}\innercustomcor}
  {\endinnercustomcor}

\begin{document}

\title[On the Moser's Bernstein Theorem]{On the Moser's Bernstein Theorem}

\author[J. E. Sampaio]{Jos\'e Edson Sampaio}
\address[Jos\'e Edson Sampaio]{Departamento de Matem\'atica, Universidade Federal do Cear\'a,
	      Rua Campus do Pici, s/n, Bloco 914, Pici, 60440-900, 
	      Fortaleza-CE, Brazil. E-mail: {\tt edsonsampaio@mat.ufc.br}}

\author[E. C. da Silva]{Eur\'ipedes Carvalho da Silva}

\address{Euripedes Carvalho da Silva: Departamento de Matem\'atica, Instituto Federal de Educa\c{c}\~ao, Ci\^encia e Tecnologia do Cear\'a,
 	      Av. Parque Central, 1315, Distrito Industrial I, 61939-140, 
 	      Maracana\'u-CE, Brazil. \newline 
 	      and 	      Departamento de Matem\'atica, Universidade Federal do Cear\'a,
	      Rua Campus do Pici, s/n, Bloco 914, Pici, 60440-900, 
	      Fortaleza-CE, Brazil.
               E-mail: {\tt euripedes.carvalho@ifce.edu.br}
 } 

\thanks{The first named author was partially supported by CNPq-Brazil grant 310438/2021-7. This work was supported by the Serrapilheira Institute (grant number Serra -- R-2110-39576).
}
 

\begin{abstract}
In this paper, we prove the following version of the famous Bernstein's theorem:
Let $X\subset \mathbb R^{n+k}$ be a closed and connected set with Hausdorff dimension $n$. Assume that $X$ satisfies the monotonicity formula at $p\in X$. Then, the following statements are equivalent:
\begin{itemize}
 \item [(1)] $X$ is an affine linear subspace;
 \item [(2)] $X$ is a definable set that is Lipschitz regular at infinity and its geometric tangent cone at infinity, $C(X,\infty)$, is a linear subspace;
 \item [(3)] $X$ is a definable set, blow-spherical regular at infinity and $C(X,\infty)$ is a linear subspace;
 \item [(4)] $X$ is a definable set that is Lipschitz normally embedded at infinity and $C(X,\infty)$ is a linear subspace;
 \item [(5)] the density of $X$ at infinity is 1.
\end{itemize}

Consequently, we prove the following generalization of Bernstein's theorem: Let $X\subset \mathbb R^{n+1}$ be a closed and connected set with Hausdorff dimension $n$. Assume that $X$ satisfies the monotonicity formula at $p\in X$ and there are compact sets $K\subset \mathbb{R}^n$ and $\tilde K\subset \mathbb{R}^{n+1}$ such that $X\setminus \tilde K$ is a minimal hypersurface that is the graph of a $C^2$-smooth function $u\colon \mathbb{R}^{n}\setminus K\to \mathbb{R}$. Assume that $u$ has bounded derivative whenever $n>7$. Then $X$ is a hyperplane.  Several other results are also presented. For example, we generalize the o-minimal Chow's theorem, we prove that any entire complex analytic set that is bi-Lipschitz homeomorphic to a definable set in an o-minimal structure must be an algebraic set. We also obtain that Yau's Bernstein Problem, which says that an oriented stable complete minimal hypersurface in $\mathbb{R}^{n+1}$ with $n\leq 6$ must be a hyperplane, holds true whether the hypersurface is a definable set in an o-minimal structure.
\end{abstract}

\maketitle

\tableofcontents

\section{Introduction}\label{introduction}

We start this article by reminding the famous Bernstein's problem:

\begin{problem_b}\label{Bernstein-problem}
If the graph of a function on $\R^{n}$ is a minimal hypersurface in $\R^{n+1}$, does this imply that the function is linear? 
\end{problem_b}


The answer is negative when $n$ is at least $8$ (see \cite{BombieriGG:1969}), but the answer was proved to be positive the cases $n=2$ by Bernstein (see \cite{Bernstein:1915}), $n = 3$ by DeGiorgi (see \cite{DeGiorgi:1965}), $n = 4$ by Almgren (see \cite{Almgren:1966}) and $n \leq 7$ by Simons (see \cite{Simons:1968}). The positive answers can be summarized as follows:

\begin{theorem}\label{thm:bernstein_thm}
Let $M\subset \R^{n+1}$ be a minimal hypersurface which is a graph of $C^2$-smooth function $u\colon \R^n\to \R$ with $2\leq n\leq 7$. Then $u$ is a linear function and $M$ must be a hyperplane. 
\end{theorem}

Furthermore, the celebrated theorem due to J. Moser in \cite{Moser:1961}, called Moser's Bernstein Theorem, says the following:
\begin{theorem}[Moser's Bernstein Theorem]\label{thm:moser_thm}
Let $M\subset \R^{n+1}$ be a minimal hypersurface which is a graph of an entire Lipschitz function $u\colon \R^n\to \R$. Then $u$ is a linear function and $M$ must be a hyperplane. 
\end{theorem}

Let us remark that, in general, Moser's Bernstein theorem does not hold in higher codimension. Lawson and Osserman in \cite{LawsonO:1977} presented a minimal cone which is the graph of a Lipschitz mapping, but is not an affine linear subspace, more precisely, they presented the following:

\begin{example}[Theorem 7.1 in \cite{LawsonO:1977}]\label{example:LO}
The graph of the Lipschitz mapping $f\colon \R^4\to \R^3$ given by $f(0)=0$ and
$$
f(x)=\frac{\sqrt{5}}{2}\|x\|\eta\left(\frac{x}{\|x\|}\right), \quad \forall x\not=0,
$$
is a minimal cone, where $\eta\colon \mathbb{S}^3\to \mathbb{S}^2$ is the Hopf mapping given by
$$
\eta(z_1,z_2)=(|z_1|^2-|z_2|^2,2z_1\bar z_2).
$$
\end{example}
Several other mathematicians approached Moser's Bernstein Theorem in higher codimension (for example, see \cite{AssimosJ:2018}, \cite{JostX:1999}, \cite{JostXY:2016}, \cite{JostXY:2018} and \cite{Xin:2005}).

In this direction, we prove the following result:
\begin{customthm}{\ref{thm:gen_moser_thm_higher}}
Let $X\subset \mathbb R^{n+k}$ be a connected closed set with Hausdorff dimension $n$. Then, $X$ is an $n$-dimensional affine linear subspace if and only if we have the following: 
\begin{enumerate}
 \item $X$ satisfies the monotonicity formula at some $p\in X$; 
 \item there are compact sets $K\subset \R^n$ and $\tilde K\subset \R^{n+k}$ such that $X\setminus \tilde K$ is the graph of a $C^1$-smooth function $u\colon \mathbb{R}^{n}\setminus K\to \mathbb{R}^{k}$ with bounded derivative;
 \item $\mathcal{N}(X,\infty)$ is a linear subspace.
\end{enumerate} 
\end{customthm}

In order to know, for a subset $X\subset \mathbb{R}^{n+k}$, with Hausdorff dimension $n$ and $p\in X$, we say that $X$ {\bf satisfies the monotonicity formula at $p$} if the function $\theta^n(X,p,\cdot)\colon (0,+\infty)\to [1,+\infty)$ given by 
$$
\theta^n(X,p,r) = \frac{\mathcal{H}^n(X\cap B_r^n(p))}{\mu_n r^n}
$$
is well-defined and nondecreasing, and satisfies the following expression:
\begin{itemize}
\item [(i)] $X$ is $C^1$-smooth at $p$ if and only if $\theta^n(X,p):=\lim \limits_{r\to 0^+}\theta^n(X,p,r)=1$;
\item [(ii)] $\theta^n(X,p,\cdot)$ is a constant function if and only if $X$ is a cone with vertex at $p$,
\end{itemize}
where $\mu_n$ is the volume of the $n$-dimensional Euclidean unit ball, $\mathcal{H}^n(A)$ is the $n$-dimensional Hausdorff measure of $A$, and $B_r^m(p)\subset \mathbb{R}^m$ is the open Euclidean ball centered at $p$ of radius $r>0$. Moreover, $\mathcal{N}(X,\infty)$ is the union of all hyperplanes $T$ such that there is a sequence $\{x_j\}_j\subset X\setminus {\rm Sing}_1(X)$ such that $\lim\|x_j\|=+\infty$ and $T_{x_j}X$ converges to $T$, where ${\rm Sing}_1(X)$ is the set of points $x\in X$ such that $X\cap U$ is not a $C^1$ smooth submanifold of $\mathbb{R}^{n+k}$, for any open neighborhood $U\subset \mathbb{R}^{n+k}$ of $x$.

\begin{remark}\label{rem:examples_monotonicity}
If $X \subset\R^{n+k}$ is a complete minimal submanifold or a closed area-minimizing set, then $X$ satisfies the monotonicity formula at any $p\in X$. In particular, any entire-pure dimensional complex analytic set $X\subset \C^n$ satisfies the monotonicity formula at any $p\in X$.
\end{remark}

In Section \ref{sec:gen_bernstein_thm}, we present several consequences of Theorem \ref{thm:gen_moser_thm_higher}. Moreover, in subsection \ref{subsec:examples_graph}, by presenting several examples, we show that Theorem \ref{thm:gen_moser_thm_higher} is sharp such that no hypothesis can be dropped.

Another natural question is whether Moser's Bernstein Theorem can be generalized to the parametric case. Thus, the following questions arise:
\begin{question}\label{question-moser-gen}
Given a smooth minimal hypersurface $M\subset \R^{n+1}$ which is bi-Lipschitz homeomorphic to $\R^n$, is $M$ an affine linear subspace?
\end{question}

Recently, \cite{FernandesS:2020} and \cite{Sampaio:2023} approached Question \ref{question-moser-gen} in higher codimension for complex analytic sets. For instance, they proved that a pure dimensional complex analytic set that is Lipschitz regular at infinity (see Definition \ref{def:Lip_regular}) must be an affine linear subset.

In this article, we use the concepts of \cite{FernandesS:2020} and \cite{Sampaio:2023} to provide partial positive answers to Question 1, even for high codimensions. We prove the following:

\begin{customthm}{\ref{main_thm}}
Let $X\subset \mathbb R^{n+k}$ be a closed and connected set with $n=\dim_H X$. Assume that $X$ satisfies the monotonicity formula at some $p\in X$. Then the following statements are equivalent:
\begin{itemize}
 \item [(1)] $X$ is an affine linear subspace;
 \item [(2)] $X$ is a definable set, blow-spherical regular at infinity and $C(X,\infty)$ is a linear subspace;
 \item [(3)] $X$ is a definable set that is Lipschitz regular at infinity and $C(X,\infty)$ is a linear subspace;
 \item [(4)] $X$ is an LNE at infinity definable set and $C(X,\infty)$ is a linear subspace;
 \item [(5)] $\theta^n(X):=\lim\limits_{r\to +\infty}\frac{\mathcal{H}^n(X\cap B_r^n(0))}{\mu_n r^n}=1$.
\end{itemize}
\end{customthm}

We obtain several consequences, for instance, we recover the results proven in \cite{FernandesS:2020} and \cite{SampaioS:2023}.

The main tool to prove Theorem \ref{main_thm} is the Kurdyka-Raby's formula at infinity. 

\begin{customthm}{\ref{thm:densitymult}}{\rm (Kurdyka-Raby's formula at infinity)}.
Let $X\subset \mathbb \R^{n+k}$ be a definable set in an o-minimal structure on $\R$ with $n=\dim X$. Let $C_1,...,C_m$ be the connected components of $C_{{\rm Smp}}(X,\infty)$. Then, $\theta^n(X)$ is defined; moreover for each $x\in \R^{n+k}$, 
$$
\theta^n(X)=\lim \limits _{r\to +\infty }\frac{\mathcal{H}^n(X\cap B_r^{n+k}(x))}{\mu_n r^n}=\sum\limits _{j=1}^mk_{X,\infty}(C_j)\frac{\mathcal{H}^{n}(C_j\cap B_1^{n+k})}{\mu_n}.
$$
\end{customthm}

We prove this formula in Section \ref{section:K-R_formula}.  Furthermore, we present some consequences of Theorem \ref{thm:densitymult}. For instance, we obtain the o-minimal Chow’s theorem proved in \cite[Corollary 4.5]{PeterzilS:2009} as follows:
\begin{customcor}{\ref{o-minimal_chow_thm}}[O-minimal Chow's theorem]
Let $X\subset\C^{n+k}$ be a pure-dimensional entire complex analytic set with $\dim_{\C}X=n$. Then, $X$ is a complex algebraic set if and only if $X$ is definable in an o-minimal structure on $\R$.
\end{customcor} 
We obtain the following generalization of the o-minimal Chow's theorem:
\begin{customcor}{\ref{lip_o-minimal_chow_thm}}[Lipschitz o-minimal Chow's theorem]
 Let $X\subset \C^{n+k}$ be a pure-dimensional entire complex analytic set with $\dim_{\C}X=n$. Then, $X$ is a complex algebraic set if and only if $X$ is lipeomorphic at infinity to a definable set in an o-minimal structure on $\R$.
\end{customcor}
We obtain also that the Yau's Bernstein Problem (see Problem 102 in \cite{Yau:1982}), which says that {\it an oriented stable complete minimal hypersurface in $\R^{n+1}$ with $n\leq 6$ must be a hyperplane}, holds true whether the hypersurface is a definable set in an o-minimal structure on $\R$ (see Corollary \ref{o-minimal_yau_conj}).


\section{Preliminaries}\label{preliminaries}
All the sets herein are assumed to be equipped with the induced Euclidean metric.

\subsection{O-minimal structures}
\begin{definition}\label{definivel}
A structure on $\R$ is a collection $\mathcal{S}=\{\mathcal{S}_n\}_{n\in\N}$ where each $\mathcal{S}_n$ is a set of subsets of $\R^n$, satisfying the following axioms:
\begin{itemize}
\item [1)] All algebraic subsets of $\R^n$ are in $\mathcal{S}_n$;
\item [2)] For every $n$, $\mathcal{S}_n$ is a Boolean subalgebra of the powerset of $\R^n$;
\item [3)] If $A\in \mathcal{S}_m$ and $B \in S_n$, then $A \times B \in \mathcal{S}_{m+n}$.
\item [4)] If $\pi \colon \R^{n+1} \to \R^n$ is the projection on the first $n$ coordinates and $A\in \mathcal{S}_{n+1}$, then $\pi(A)\in \mathcal{S}_n$.
\end{itemize}
The structure $\mathcal{S}$ is said to be {\bf o-minimal} if it satisfies the following condition:
\begin{itemize}
\item [5)] The elements of $\mathcal{S}_1$ are precisely finite unions of points and intervals.
\end{itemize}
A element of $\mathcal{S}_n$ is called {\bf definable in $\mathcal{S}$}.
\end{definition}

Throughout this paper, we fix an o-minimal structure $\mathcal{S}$ on $\R$.
\begin{definition}
A mapping $f\colon A\subset \R^n \to \R^m$ is called {\bf definable in $\mathcal{S}$} if its
graph is an element of $\mathcal{S}_{n+m}$. 
\end{definition} 
In the sequel, the adjective definable denotes definable in $\mathcal{S}$.

\subsection{Dimension of definable sets}
\begin{definition}
Let $X\subset \R^m$ be a subset and $k$ be a positive integer. The {\bf $C^k$ singular set} of $X$, denoted by ${\rm Sing}_k(X)$, is the set of points $x\in X$ such that $U\cap X$ is not a smooth submanifold of $\R^m$ for any open neighbourhood $U$ of $x$. A point of ${\rm Sing}_k(X)$ is called a {\bf $C^k$ singular point} (or a {\bf $C^k$ singularity}) of $X$. If $p\in {\rm Reg}_k(X):=X\setminus {\rm Sing}_k(X)$, we say that $X$ is smooth at $p$.
\end{definition}
Thus, if $p\in {\rm Reg}_k(X)=X\setminus {\rm Sing}_k(X)$, there is open neighbourhood $U\subset\R^m$ of $p$ such that $X\cap U$ is a $C^k$ smooth submanifold of $\R^m$ and, then, we define the {\bf real dimension of $X$ at $p$} by $\dim_p X=\dim X\cap U$. Thus, we define the {\bf dimension of $X$} by 
$$
\dim X =\max\limits_{p\in {\rm Reg}_1(X)} {\dim}_p X.
$$

We say that $X$ is a {\bf pure-dimensional set}, if $\dim X = \dim_p X$ for all $p\in {\rm Reg}_1(X)$.

In the case that $X\subset \R^m$ is a definable set, we have that ${\rm Sing}_k(X)$ is also a definable set with $\dim {\rm Sing}(X)<\dim X$ and $\dim X$ is equal to the Hausdorff dimension of $X$, denoted here by $\dim_H X$. 

In the case that $X\subset \C^m\cong \R^{2m}$ is a complex analytic set, we define the {\bf complex dimension of $X$} as $\dim_{\C}X=\frac{1}{2}\dim X$.

\subsection{Lipschitz regularity at infinity}
\begin{definition}\label{bilipschitz equivalence}
Let $X\subset \R^n$ and $Y\subset\R^m$ be two subsets. 
A mapping $\phi \colon X\rightarrow Y$ is {\bf Lipschitz} if there is a constant $C\geq 0$ such that
$$
\|\phi(x)-\phi(y)\|\leq C\|x-y\|, \quad \forall x,y\in X.
$$
In this case, we also say that $\phi$ is $C$-Lipschitz.
Moreover, a mapping $\phi \colon X\rightarrow Y$ is a {\bf lipeomorphism} if $\phi$ is a homeomorphism such that $\phi$ and $\phi^{-1}$ are Lipschitz.
When there is a lipeomorphism $\phi \colon X\rightarrow Y$, we say that $X$ and $Y$ are {\bf lipeomorphic}.
\end{definition}
\begin{definition}\label{bilipschitz_equiv_infinity}
Let $X\subset \R^n$ and $Y\subset\R^m$ be two subsets. We say that $X$ and $Y$ are {\bf lipeomorphic at infinity}, if there exist compact subsets $K\subset\R^n$ and $\widetilde K\subset \R^m$ such that $X\setminus K$ and $Y\setminus \widetilde K$ are lipeomorphic.
\end{definition}

\begin{definition}\label{def:Lip_regular}
A subset $X\subset\R^m$ is called {\bf Lipschitz regular at infinity} if $X$ and $\R^n$ are lipeomorphic at infinity for some $n\in\N$.
\end{definition}

\subsection{Lipschitz normal embedding at infinity}\label{inner distance}
Let us recall the definition of the inner distance. Given a path connected subset $X\subset\R^m$ the
\emph{inner distance}  on $X$  is defined as follows: given two points $x_1,x_2\in X$, $d_X(x_1,x_2)$  is the infimum of the lengths of paths on $X$ connecting $x_1$ to $x_2$. As stated in Section \ref{preliminaries}, all the sets considered herein are supposed to be equipped with the induced Euclidean metric. Whenever we consider the inner distance, we emphasize it clearly.

\begin{definition}[See \cite{BirbrairM:2000}]\label{normal embedding} A subset $X\subset\R^m$ is called {\bf Lipschitz normally embedded} (or shortly {\bf LNE}) if there exists $C >0$ such that
$$d_X(x_1,x_2)\le C \|x_1-x_2\|$$
for all $x_1,x_2\in X$. In this case, we also say that $X$ is {\bf $C$-LNE}.
\end{definition}


In \cite{FernandesS:2020}, the following definition was presented:

\begin{definition}\label{normal embedding at infinity} A subset $X\subset \R^n$ is {\bf Lipschitz normally embedded at infinity} (or shortly {\bf LNE at infinity}) if there exists a compact subset $K\subset \R^n$ such that $X\setminus K$ is Lipschitz normally embedded.
\end{definition}

\subsection{Tangent cones}
\begin{definition}
Let $X\subset \R^{m}$ be an unbounded set (resp. $p\in \overline{X}$ be a non-isolated point). We say that $v\in \R^{m}$ is a tangent vector of $X$ at infinity (resp. $p$) if there are a sequence of points $\{x_i\}\subset X$ tending to infinity (resp. $p$) and a sequence of positive real numbers $\{t_i\}$ such that 
$$\lim\limits_{i\to \infty} \frac{1}{t_i}x_i= v \quad (\mbox{resp. } \lim\limits_{i\to \infty} \frac{1}{t_i}(x_i-p)= v).$$
Let $C(X,\infty)$ (resp. $C(X,p)$) denote the set of all tangent vectors of $X$ at infinity (resp. $p$). We call $C(X,\infty)$ the {\bf geometric tangent cone of $X$ at infinity} (resp. $p$). 
\end{definition}

Another way to present the geometric tangent cone at infinity (resp. $p$) of a subset $X\subset\R^{m}$ is via the spherical blow-up at infinity (resp. $p$) of $\R^{m}$. Let us consider the {\bf spherical blowing-up at infinity} (resp. $p$) of $\R^{m}$, $\rho_{\infty}\colon\mathbb{S}^{m-1}\times (0,+\infty )\to \R^{m}$ (resp. $\rho_p\colon\mathbb{S}^{m-1}\times [0,+\infty )\to \R^{m}$), given by $\rho_{\infty}(x,r)=\frac{1}{r}x$ (resp. $\rho_p(x,r)=rx+p$).

Note that $\rho_{\infty}\colon\mathbb{S}^{m-1}\times (0,+\infty )\to \R^{m}\setminus \{0\}$ (resp. $\rho_p\colon\mathbb{S}^{m-1}\times (0,+\infty )\to \R^{m}\setminus \{0\}$) is a homeomorphism with inverse mapping $\rho_{\infty}^{-1}\colon\R^{m}\setminus \{0\}\to \mathbb{S}^{m-1}\times (0,+\infty )$ (resp. $\rho_{p}^{-1} \colon\R^{m}\setminus \{0\}\to \mathbb{S}^{m-1}\times (0,+\infty ) $) given by $\rho_{\infty}^{-1}(x)=(\frac{x}{\|x\|},\frac{1}{\|x\|})$ (resp. $\rho_p^{-1}(x)=(\frac{x-p}{\|x-p\|},\|x-p\|)$).

 The {\bf strict transform} of the subset $X$ under the spherical blowing-up $\rho_{\infty}$ (resp. $\rho_{p}$) is $X'_{\infty}:=\overline{\rho_{\infty}^{-1}(X\setminus \{0\})}$ (resp. $X'_{p}:=\overline{\rho_{p}^{-1}(X\setminus \{0\})}$). The subset $X_{\infty}'\cap (\mathbb{S}^{m-1}\times \{0\})$ (resp. $X_{p}'\cap (\mathbb{S}^{m-1}\times \{0\})$) is called the {\bf boundary} of $X'_{\infty}$ (resp. $X'_p$) and is denoted as $\partial X'_{\infty}$ (resp. $\partial X'_p$). 

\begin{remark}\label{remarksimplepointcone}
{\rm If $X\subset \R^{m}$ is a definable set in an o-minimal structure, then $\partial X'_{\infty}=(C(X,\infty)\cap \mathbb{S}^{m-1})\times \{0\}$ (resp. $\partial X'_{p}=(C(X,p)\cap \mathbb{S}^{m-1})\times \{0\}$.}
\end{remark}

We finish this Subsection by reminding the invariance of the tangent cone at infinity under lipeomorphisms at infinity obtained in the paper \cite{SampaioS:2022} (see also \cite[Theorem 2.19]{FernandesS:2020}). 

\begin{theorem}[Corollary 2.11 in \cite{SampaioS:2022}]\label{inv_cones}
Let $X\subset\R^{m}$ and $Y\subset\R^{k}$ be unbounded definable sets with $n=\dim X=\dim Y$. If there exist compact subsets $K\subset \mathbb{R}^{m}$ and $\tilde K\subset \mathbb{R}^{k}$, and a lipeomorphism $\varphi\colon X\setminus K\to Y\setminus \tilde K$, then there is a globally defined lipeomorphism $d\varphi\colon C(X,\infty)\to C(Y,\infty)$ with $d\varphi(0)=0$.
\end{theorem}

\subsection{Relative multiplicities at infinity}

Let $X\subset \R^{m}$ be an $n$-dimensional definable subset in an o-minimal structure and $p\in \R^{m}\cup \{\infty\}$. We say that  $x\in\partial X'_p$ is {\bf simple point of $\partial X'_p$}, if there is an open subset $U\subset \R^{m+1}$ with $x\in U$ such that:
\begin{itemize}
\item [a)] the connected components $X_1,\cdots , X_r$ of $(X'_p\cap U)\setminus \partial X'_{p}$ are topological submanifolds of $\R^{m+1}$ with $\dim X_i=\dim X$, for all $i=1,\cdots,r$;
\item [b)] $(X_i\cup \partial X'_{p})\cap U$ are topological manifolds with boundary $\partial X'_{p}\cap U$, for all $i=1,\cdots ,r$. 
\end{itemize}

Let ${\rm Smp}(\partial X'_p)$ be the set of simple points of $\partial X'_p$ and we define $C_{{\rm Smp}}(X,p)=\{t\cdot x; \, t>0\mbox{ and }(x,0)\in {\rm Smp}(\partial X')\}$. Let $k_{X,p}\colon C_{{\rm Smp}}(X,p)\to \N$ be the function such that $k_{X,p}(x)$ is the number of connected components of the germ $(\rho^{-1}(X\setminus\{p\}),v)$, where $v=(\frac{x}{\|x\|},0)$.

\begin{remark}\label{remarkdense}
	${\rm Smp}(\partial X'_p)$ is an open dense subset of the $(n-1)$-dimensional part of $\partial X'_p$ whenever $\partial X'_p$ is an $(n-1)$-dimensional subset, where $n=\dim X$. 
\end{remark} 
\begin{definition}\label{def:relative_mult}
It is clear that the function $k_{X,p}$ is locally constant. In fact, $k_{X,p}$ is constant for each connected component $X_j$ of $C_{{\rm Smp}}(X,p)$. We define {\bf the relative multiplicity at $p$ of $X$ (along of $X_j$)} as $k_{X,p}(X_j):=k_{X,p}(x)$ where $x\in X_j$. 
\end{definition}

\begin{definition}
Let $X\subset \mathbb{R}^m$ and $Y\subset \mathbb{R}^k$ be closed sets. Let $p\in \mathbb{R}^m\cup \{\infty\}$, $q\in \mathbb{R}^k\cup \{\infty\}$ and a homeomorphism $\varphi:X\rightarrow Y$ such that $q=\lim\limits_{x\rightarrow p}{\varphi(x)}$, is said a {\bf blow-spherical homeomorphism at $p$}, if  the homeomorphism 
$$\rho_{q}^{-1}\circ\varphi\circ \rho_p\colon \rho_{p}^{-1}(X\setminus\{p\})\rightarrow \rho_{q}^{-1}(Y\setminus\{q\})$$
extends to a homeomorphism $\varphi'\colon X'_p\rightarrow Y'_q$. 
\end{definition}

\subsection{Blow-spherical invariance of the relative multiplicities at infinity}

The following result was presented in \cite[Proposition 3.5]{SampaioS:2023} for semialgebraic sets, but the proof is the same for definable sets in an o-minimal structure as we can see in \cite[Teorema 3.1.7]{Sampaio:2015}.

\begin{proposition}\label{propinvarianciamultiplicidaderelativa}
	Let $X \subset\mathbb{R}^m$ and $Y\subset\mathbb{R}^k$ be definable subsets in an o-minimal structure on $\R$. Let $\varphi\colon X\rightarrow Y$ be a blow-spherical homeomorphism at $p\in \mathbb{R}^n\cup \{\infty\}$. Then $\varphi'(Smp(\partial X'_p))=Smp(\partial Y'_q)$ and 
	$$k_{X,p}(x)=k_{Y,q}(\varphi'(x)),$$
	for all $x\in Smp(\partial X'_p)$, where $q=\lim\limits_{x\rightarrow p}{\varphi(x)}$.
\end{proposition}

\subsection{Bi-Lipschitz invariance of the relative multiplicities at infinity}

\begin{theorem}\label{multiplicities}
Let $X\subset\R^{m}$ and $Y\subset\R^{k}$ be unbounded definable sets with $n=\dim X=\dim Y$. If there exist compact subsets $K\subset \mathbb{R}^{m}$ and $\tilde K\subset \mathbb{R}^{k}$, and a lipeomorphism $\varphi\colon X\setminus K\to Y\setminus \tilde K$, then there exists a lipeomorphism $d\varphi\colon C(X,\infty)\to C(Y,\infty)$ that satisfies 
$$
k_{X,\infty}(x)=k_{Y,\infty}(d\varphi(x)),\quad \forall x\in C_{{\rm Smp}}(X,\infty)\cap d\varphi^{-1}(C_{{\rm Smp}}(Y,\infty)).
$$
In particular, $d\varphi|_{C_{{\rm Smp}}(X,\infty)}\colon C_{{\rm Smp}}(X,\infty)\to C_{{\rm Smp}}(Y,\infty)$ is a lipeomorphism where $C_{{\rm Smp}}(X,\infty)\not=\emptyset $.
\end{theorem}

\begin{proof}
We closely follow the proof of Theorem 4.2 in \cite{Sampaio:2021}, but because its importance in this article, we present the proof here.

By making identifications $X\setminus K\longleftrightarrow X\setminus K \times\{0\}$ and $Y\setminus \tilde K\longleftrightarrow \{0\}\times Y\setminus \tilde K$, we can assume that $m=k$ and $\varphi$ is globally defined in $\R^m$. 
We know that there exist a sequence of positive real numbers $S=\{t_j\}_{j\in\N}$ and a lipeomorphism $d\varphi\colon C(X,\infty)\to C(Y,\infty)$ such that
$$t_j\to +\infty \quad\mbox{and} \quad \frac{\varphi(t_jv)}{t_j}\to d\varphi(v), \quad \forall v\in C(X,\infty).$$
For more details, see \cite{Sampaio:2016}, \cite{FernandesS:2020}, \cite{FernandesS:2023} and \cite{SampaioS:2022}.

Thus, $\rho^{-1}\circ\varphi\circ\rho \colon SX\rightarrow Y'$ is an injective and continuous mapping that continuously extends to a mapping $\varphi'\colon \overline{SX}\to Y',$ where $\rho=\rho_{\infty}$ and $$SX=\{(x,s)\in \mathbb{S}^{m-1}\times (0,+\infty);\,\textstyle\frac{1}{s}\cdot x \in X \mbox{ and }\frac{1}{s}\in S\}.$$

We note that $Smp(\partial X'_{\infty})=\emptyset$ (resp. $Smp(\partial Y'_{\infty})=\emptyset$) if and only if $\dim C(X,\infty)<\dim X$ (resp. $\dim C(Y,\infty)<\dim Y$). Since by Theorem \ref{inv_cones}, $\dim C(X,{\infty})=\dim C(Y,{\infty})$, then we obtain that $Smp(\partial X'_{\infty})=\emptyset$ if and only if $Smp(\partial Y'_{\infty})=\emptyset$. However, when $Smp(\partial X'_{\infty})=Smp(\partial Y'_{\infty})=\emptyset$, we have nothing to do. Thus, we can assume that $Smp(\partial X'_{\infty})\not=\emptyset$ and, thus, $Smp(\partial Y'_{\infty})\not=\emptyset$ as well. 
Further, $C_{{\rm Smp}}(X,{\infty})$ (resp. $C_{{\rm Smp}}(Y,{\infty})$)  is a dense subset in the $n$-dimensional part of $C(X,{\infty}) $ (resp. $C(Y,{\infty})$). Therefore, $d\varphi (C_{{\rm Smp}}(X,{\infty}))\cap C_{{\rm Smp}}(Y,{\infty})$ is a dense subset in the $n$-dimensional part of $C(Y,{\infty})$ and $C_{{\rm Smp}}(Y,{\infty})$.

Let $X_1,\cdots, X_r$ be the connected components of $C_{{\rm Smp}}(X,{\infty})$ and let $Y_1,\cdots, Y_s$ be the connected components of $C_{{\rm Smp}}(Y,{\infty})$. For each point $x\in X_j$, we know that $k_{X,{\infty}}(X_j)=k_{X,{\infty}}(x)$ is the number of connected components of the germ $(\rho^{-1}_{\infty}(X)\cap B_{\delta }^{m+1}(x),x)$. Then, $k_{X,{\infty}}(x)$ can be seen as the number of connected components of the set $(SX\cap \mathbb{S}^{m-1}\times \{t_j\})\cap B_{\delta }^{m+1}(x)$ for all sufficiently large $k$.

Let $\pi\colon\R^{m}\to \R^n$ be a linear projection such that 
$$\pi^{-1}(0)\cap(C(X,{\infty})\cup C(Y,{\infty}))=\{0\}.$$ 
Let ramification loci of 
$$\pi_{| X}\colon X\to \R^n \quad \mbox{and} \quad \pi_{| C(X,{\infty})}\colon C(X,{\infty})\to \R^n$$ 
be $\sigma(X)$ and $\sigma(C(X,{\infty}))$ respectively. Similarly, we define $\sigma(Y)$ and $\sigma(C(Y,{\infty}))$.

Let $Z=\sigma(X)\cup \sigma(C(X,{\infty}))$, $W=\sigma(Y)\cup \sigma(C(Y,{\infty}))$, and 
$$
\Sigma=Z\cup C(Z,{\infty})\cup \pi(d\varphi^{-1} (\pi|_{Y\cup C(Y,{\infty})}^{-1}(W\cup C(W,{\infty}))).
$$ 
Since $\pi|_{Y\cup C(Y,{\infty})}$ is a proper mapping, $d\varphi$ is a lipeomorphism, and $\dim Z\cup W\leq n-1$, we obtain that $\dim \Sigma\leq n-1$. Thus, we obtain that if $v'\in \R^n\setminus \Sigma$ then $\pi|_X^{-1}(v')\subset \R^{m}\setminus \pi|_X^{-1}(Z\cup C(Z,{\infty}))$ and for any $v_i\in \pi|_X^{-1}(v')$, we have that $w'_i=\pi(d\varphi(v_i))\in \R^n\setminus (W\cup C(W,{\infty}))$. For $\eta,R >0$, we define the following set 
$$
C_{\eta,R }(v')=\{w\in \R^n|\, \exists t>0; \|tv'-w\|<\eta t\}\setminus B_{R }^{n}(0).
$$

Let $\eta,R >0$ be sufficiently large such that 
$C_{\eta,R }(v')\subset \R^n\setminus (Z\cup C(Z,{\infty}))$.
Then, we obtain that $\pi|_V\colon V\to C_{\eta,R }(v')$ is a lipeomorphism for each connected component $V$ of  $\pi^{-1}(C_{\eta,R}(v'))\cap X$. Therefore, for each $j=1,\dots,r$, there are different connected components $V_{j1},\dots,V_{jk_X(X_j)}$ of $\pi^{-1}(C_{\eta,R }(v'))\cap X$ such that 
$$C(\overline{V_{ji}},{\infty})\setminus \{0\}\subset X_j, \quad \forall i\in \{1,...,k_{X,{\infty}}(X_j)\}.$$

Let $\rho_{\infty}\colon\mathbb{S}^{m-1}\times (0,+\infty )\to \R^{m}$ be the spherical blow-up at infinity. Fixed $j\in\{1,\cdots,r\}$, let us suppose that there is $q\in\{1,\cdots,s\}$ such that $d\varphi (X_j)\cap Y_q\not=\emptyset$ and $k_{X,\infty}(X_j)>k_{Y,\infty}(Y_q)$, it means that if we consider a point $x=(v,0)\in (X_j\cap \mathbb{S}^{m-1})\times \{0\}$ with $d\varphi(x)\in Y_q$ and, since $\dim \Sigma \leq n-1$, we can assume that $\pi(v)\not\in \Sigma$, then there are at least two different connected components $V_{ji}$ and $V_{jl}$ of $\pi^{-1}(C_{\eta,R }(\pi(v)))\cap X$ and sequences $\{(x_j,t_j^{-1})\}_{j\in \N}\subset \rho^{-1}(V_{ji})\cap SX$ and $\{(y_j,t_j^{-1})\}_{j\in \N}\subset \rho^{-1}(V_{jl})\cap SX$ such that $\lim (x_j,t_j^{-1})=\lim (y_j,t_j^{-1})=x$ and $\varphi'(x_j,t_j^{-1}),\varphi'(y_j,t_j^{-1})\in \rho^{-1}(\widetilde V_{q\ell})$ for all $j\in \N$, where $\widetilde V_{q\ell}$ is a connected component of  $\pi^{-1}(C_{\tilde\eta,\tilde R }(\pi(d\varphi(v))))\cap Y$ for some $\tilde\eta,\tilde R >0$ such that 
$$
C_{\tilde\eta,\tilde R }(w')=\{\tilde w\in\R^n|\, \exists t>0; \|tw'-\tilde w\|<\tilde\eta t\}\setminus B_{\tilde R }^{n}(0)
$$
satisfies $C_{\tilde\eta,\tilde R }(w')\subset \R^n\setminus (W\cup C(W,{\infty}))$, where $w'=\pi(d\varphi(v))$. In particular,  $\pi|_{\widetilde V}\colon \widetilde V\to C_{\tilde\eta,\tilde R }(w')$ is a lipeomorphism, where $\widetilde V=\widetilde V_{q\ell}$.
Since $\varphi(t_jx_j),\varphi(t_jy_j)\in\widetilde V$ for all $j\in \N$, we have 
$$\|\varphi(t_jx_j)-\varphi(t_jy_j)\|=o_{\infty}(t_j)$$ 
and
$$d_{Y}(\varphi(t_jx_j),\varphi(t_jy_j))\leq d_{\widetilde V}(\varphi(t_jx_j),\varphi(t_jy_j))=o_{\infty}(t_j),$$
where $g(t_j)=o_{\infty}(t_j)$ means that $\lim\limits_{j\to +\infty}\frac{g(t_j)}{t_j}=0$.
Now, since  $X$ is lipeomorphic at infinity to $Y$, we obtain $d_{X}(t_jx_j,t_jy_j)\leq o_{\infty}(t_j)$.
On the other hand, since $t_jx_j$ and $t_jy_j$ lie in different connected components of $\pi^{-1}(C_{\eta,R }(\pi(v)))\cap X$, there exists a constant $C>0$ such that $d_{X}(t_jx_j,t_jy_j)\geq Ct_j$, which is a contradiction.

We have proved that $k_{X,{\infty}}(X_j)\leq k_{Y,{\infty}}(Y_q)$. By similar arguments, by using that $\varphi^{-1}$ is a lipeomorphism, we can also prove $k_{Y,{\infty}}(Y_q)\leq k_{X,{\infty}}(X_j)$. Therefore, $k_{X,{\infty}}(X_j)=k_{Y,{\infty}}(Y_q)$ for any $q\in\{1,\cdots,s\}$ such that $d\varphi (X_j)\cap Y_q\not=\emptyset$. Then
$$
k_{X,\infty}(x)=k_{Y,\infty}(d\varphi(x)),\quad \forall x\in C_{{\rm Smp}}(X,\infty)\cap d\varphi^{-1}(C_{{\rm Smp}}(Y,\infty)).
$$
\end{proof}

\begin{remark}
Since the inversion mapping $\iota\colon \R^{m}\setminus\{0\}\to \R^{m}\setminus\{0\}$, given by $\iota(x)=\frac{x}{\|x\|^2}$, is a blow-spherical homeomorphism at infinity, for $Z=\iota(X\setminus \{0\})$, we obtain that $C(Z,0)=C(X,\infty)$, $C_{\rm Smp}(X,\infty)=C_{\rm Smp}(Z,0)$ and $k_{X,\infty}(v)=k_{Z,0}(v)$, for all $v\in C_{\rm Smp}(X,\infty)$, and thus the o-minimal version of \cite[Theorem 4.2]{Sampaio:2021} follows from Theorem \ref{multiplicities} and \cite[Theorem 4.1]{Sampaio:2023b}.
\end{remark}

\subsection{Density at infinity}
Let $X\subset \mathbb{R}^{m}$ be a definable set. Fixed $n\in \mathbb{N}$ and $p\in X$, we define the function $\theta^n(X,p,\cdot)\colon (0,+\infty)\to X$ as follows:
$$
\theta^n(X,p,r) = \frac{\mathcal{H}^n(X\cap B_r^n(p))}{\mu_n r^n},
$$
where $\mu_n$ is the volume of the $n$-dimensional Euclidean unit ball, $\mathcal{H}^n(A)$ is the $n$-dimensional Hausdorff measure of $A$ and $B_r^m(p)\subset \mathbb{R}^m$ is the open Euclidean ball centered at $p$ of radius $r>0$. 

For simplicity, we denote $B_r^m:=B_r^m(0)$.
\begin{definition}
Let $X\subset \mathbb{R}^{n+k}$ be a set. We say that $X$ has {\bf $n$-dimensional density at $p$} and we denote it by $\theta^n(X,p)$, if the limit exists:
$$
\theta^n(X,p):= \lim_{r\rightarrow 0^+}{\theta^n(X,p,r)}.
$$
\end{definition}

\begin{definition}
Let $X\subset \mathbb{R}^{n+k}$ be a set. We say that $X$ {\bf has $n$-dimensional density at infinity} and we denote it by $\theta^n(X)$, if the limit exists:
$$
\theta^n(X):= 
\lim_{r\rightarrow +\infty}\frac{\mathcal{H}^n(X\cap B_r^{n+k})}{\mu_n r^n}.
$$
When $X$ has $n$-dimensional density at infinity and $n=\dim X$, we say that $X$ {\bf has density at infinity}.
\end{definition}

\section{Kurdyka-Raby's formula at infinity}\label{section:K-R_formula}

The goal of this section is to bring the Theorem of Kurdyka and Raby in \cite{Kurdyka:1989} and its ideas to a global perspective. Then, we prove the following Kurdyka-Raby's formula at infinity:

\begin{theorem}[Kurdyka-Raby's formula at infinity]\label{thm:densitymult}
Let $X\subset \mathbb \R^{n+k}$ be a definable set in an o-minimal structure on $\R$ with $n=\dim X$. Let $C_1,...,C_m$ be the connected components of $C_{{\rm Smp}}(X,\infty)$. Then, $\theta^n(X)$ is defined and for each $x\in \R^{n+k}$, 
$$
\theta^n(X)=\lim \limits _{r\to +\infty }\frac{\mathcal{H}^n(X\cap B_r^{n+k}(x))}{\mu_n r^n}=\sum\limits _{j=1}^mk_{X,\infty}(C_j)\frac{\mathcal{H}^{n}(C_j\cap B_1^{n+k})}{\mu_n}.
$$
\end{theorem}

\subsection{Proof of the Kurdyka-Raby's formula at infinity}
\begin{proof}[Proof of Theorem \ref{thm:densitymult}]
We first prove that the result is true for open definable sets.
\begin{claim}\label{propdensitycone}
Let $\Omega \subset \mathbb{R}^{n}$ be an open definable set. Then, the density at infinity satisfies
$$\lim_{r\rightarrow +\infty}{\frac{\mathcal{H}^{n}\left(\Omega\cap B_r^{n}\right)}{r^{n}}}=\mathcal{H}^{n}\left(B_1^{n}\cap C(\Omega,\infty)\right).$$
\end{claim}
\begin{proof}[Proof of the Claim \ref{propdensitycone}]
First, note that $\theta^{n}(\Omega)$ there exists if and only if there exists $\theta^{n}(\Omega\setminus K)$, where $K\subset \Omega$ is a compact subset. Moreover, in this case, $\theta^{n}(\Omega)=\theta^{n}(\Omega\setminus K)$. In fact, 

\begin{eqnarray*}
	\theta^{n}(\Omega) & = & \lim_{r\rightarrow +\infty}{\frac{\mathcal{H}^{n}(\Omega\cap B_r^{n})}{\mu_nr^{n}}}\\
	& = &  \lim_{r\rightarrow +\infty}{\frac{\mathcal{H}^{n}(\Omega\cap K\cap B^{n}_r) + \mathcal{H}^{n}((\Omega\setminus K)\cap B^{n}_r)}{\mu_nr^{n}}}\\
	& = & \lim_{r\rightarrow +\infty}{\frac{\mathcal{H}^{n}((\Omega\setminus K)\cap B^{n}_r)}{\mu_nr^{n}}}\\
	& = & \theta^{n}(\Omega\setminus K).
\end{eqnarray*}

When $C$ is a cone with the vertex at $0$ in $\mathbb{R}^{n}$, we have
$$\mathcal{H}^{n}\left(C\cap B_r^{n}\right) = r^{n}\mathcal{H}^{n}\left(C\cap B^{n}_1\right),$$
for every $r>0$. Now, 
we consider the cone $C_{r}:=\mbox{Cone}_0\left(\Omega \setminus B_r^{n}\right)$.

\begin{claim}\label{claimone}
	$\bigcap_{r>0}{\overline{C_r}}=C(\Omega,\infty)$.
\end{claim}
\begin{proof}[Proof of the Claim \ref{claimone}]

Let $v\in \bigcap_{r>0}{\overline{C_r}}$. Then, $v\in C_k$ for every $l\in \mathbb{N}$. Therefore, $v=t_lx_l$, where $x_l\in \Omega\setminus B^{n}_{l}$, with $\|x_l\|\geq l$, and $t_l\in (0,+\infty)$. Hence, $v\in C(\Omega,\infty)$.

Conversely, if $v\in C(\Omega,\infty)$, it follows from \cite[Proposition 2.15]{FernandesS:2020} that there exists a continuous and definable curve $\alpha\colon (\epsilon,+\infty)\rightarrow \Omega$ such that $\alpha(t)=tv+o_{\infty}(t)$ with $\|\alpha(t)\|>t$, where $g(t)=o_{\infty}(t)$ means that $\lim\limits_{t\to +\infty}\frac{g(t)}{t}=0$. Now, define $\beta(t)=\frac{\alpha(t)}{t}=v+\frac{o_{\infty}(t)}{t}$, and note that $\|\beta(t)-v\|\rightarrow 0$ as $t\rightarrow +\infty$. Since $\alpha(t)\subset C_t$, we have $\frac{\alpha(t)}{t}\in C_t$, and therefore $\beta(t)\in C_t$. Thus,
$$\text{dist}(v,C_t)\leq \|v-\beta(t)\|,$$
for every $t\geq 0$. However, $C_t\subset C_s$ when $s<t$. Then for any $s>0$ and $s<t$, we have 
$$\text{dist}(v,\overline{C_s})\leq \text{dist}(v,\overline{C_t})\leq \|v-\beta(t)\|\rightarrow 0,$$
as $t\rightarrow +\infty$. Hence, $v\in \overline{C_{s}}$ for any $s>0$ and, consequently, $v\in \bigcap_{r>0}{\overline{C_r}}$.
\end{proof}

Now, consider $Z=\bigcap_{r>0}{C_{r}}$, and it follows from the claim that $C(\Omega,\infty)\supset Z$.

Moreover, the functions $r \mapsto \mathcal{H}^{n}\left(C_{r}\cap B_1^{n}\right)$ and $r \mapsto \mathcal{H}^{n}\left(\overline{C_{r}}\cap B_1^{n}\right)$ are nonincreasing; hence, the following limits exist:
\begin{eqnarray*}
	\lim_{r\rightarrow +\infty}{\frac{\mathcal{H}^{n}\left(C_{r}\cap B^{n}_r\right)}{r^{n}}} & = & \lim_{r \rightarrow +\infty}{\mathcal{H}^{n}\left(C_{r}\cap B^{n}_1\right)}\\
	& = & \mathcal{H}^{n}\left(B^{n}_1\cap \bigcap_{r>0} C_{r}\right)\\
	& = & \mathcal{H}^{n}\left(B^{n}_1\cap Z\right)
\end{eqnarray*}
and

\begin{eqnarray*}
	\lim_{r\rightarrow +\infty}{\frac{\mathcal{H}^{n}\left(\overline{C_{r}}\cap B^{n}_r\right)}{r^{n}}} & = & \lim_{r \rightarrow +\infty}{\mathcal{H}^{n}\left(\overline{C_{r}}\cap B^{n}_1\right)}\\
	& = & \mathcal{H}^{n}\left(B^{n}_1\cap \bigcap_{r>0} \overline{C_{r}}\right)\\
	& = & \mathcal{H}^{n}(B^{n}_1\cap C(\Omega,\infty)).
\end{eqnarray*}

Since $C_{r}$ is a definable set, we have
$$\mathcal{H}^{n}\left(B_1^{n} \cap C_{r} \right)=\mathcal{H}^{n}\left(B_1^{n} \cap \overline{C}_{r} \right).$$

Therefore, $\mathcal{H}^{n}(B^{n}_1\cap C(\Omega,\infty))=\mathcal{H}^{n}(B^{n}_1\cap Z)$.

Moreover, for every $R>0$, we have $C_{R} \supset \Omega \setminus (\Omega\cap B^{n}_R)$. Thus,

\begin{eqnarray}\label{des.over}
\nonumber	\limsup_{r\rightarrow +\infty}{\frac{\mathcal{H}^{n}\left(\Omega\cap B^{n}_r\right)}{r^{n}}} & = & \limsup_{r\rightarrow +\infty}{\frac{\mathcal{H}^{n}\left((\Omega\setminus (\Omega\cap B^{n}_R))\cap B^{n}_r\right)}{r^{n}}}\\
	\nonumber& \leq & \lim_{r\rightarrow +\infty}{\frac{\mathcal{H}^{n}\left(C_r\cap B^{n}_r\right)}{r^{n}}}\\
	\nonumber& = & \mathcal{H}^{n}(Z\cap B^{n}_1)=\mathcal{H}^{n}(B^{n}_1\cap C(\Omega,\infty)).
\end{eqnarray}

For $w\in \R^{n}\setminus \{0\}$, we set $(w,+\infty):=\{tw;t>1\}$.
Now, take $z\in Z\setminus \{0\}$. Without any loss of generality, we assume that $0\not\in \Omega$.
Then, $(z,+\infty)\cap \Omega$ is a definable set in $\Omega$, and we define $\alpha(z)=\sup \left\{\lambda\geq 0; \lambda z\notin \Omega\right\}$, which is nonnegative. If $\alpha(z)\not =0$, the segment $(\alpha(z)z,+\infty)$ is contained in $\Omega$ and if $\alpha(z)=0$, then $(\epsilon z,+\infty)$ is contained in $\Omega$, for all $\epsilon>0$.

We define the following set
$$Z_{r}=\{z\in Z\setminus \{0\}/ \|\alpha(z)\cdot z\|\leq r\}.$$

Since $\alpha(\mu z)=\sup\{\lambda\geq 0; \lambda (\mu z)\notin \Omega\}=\lambda_0$ and $\alpha(z)=\sup\{\eta\geq 0; \eta z\notin \Omega\}=\eta_0$, we have $\lambda_0=\frac{\eta_0}{\mu}$, and $Z_r$ is indeed a cone.

Now, note that $Z_{r}\setminus \overline{B^{n}_r}\subset \Omega\setminus \overline{B^{n}_r}$. Indeed, if $z\in Z_{r}\setminus \overline{B^{n}_r}$, then $\|z\|>r$, and additionally, $\|\alpha(z)z\|\leq r$. Therefore, $\alpha(z)\leq \frac{r}{\|z\|}<1$. Note that for $r<s$, we have $Z_r\subset Z_s$. In fact, if $z\in Z_r$, then $\alpha(z)\|z\|\leq r<s$, so $z\in Z_s$.

\begin{claim}\label{claimtwo}
	$Z\setminus \{0\}=\bigcup_{r>0}{Z_r}$.
\end{claim}
\begin{proof}[Proof of the Claim \ref{claimtwo}]
It follows from the definition that $\bigcup_{r>0}{Z_r}\subset Z\setminus \{0\}$. Reciprocally, assume that $z\in Z\setminus \{0\}$. Thus for each $r>0$ there exists a points $z_r\in \Omega\setminus B^{n}_r$ and $\lambda_r\in (0,+\infty)$ such that $\lambda_r z_r=z$ and since $\Omega \setminus \overline{B^{n}_r}$ is open set there exists $t_{r_j}\in (0,+\infty)$ such that $t_{r_j}z_j\in \Omega\setminus \overline{B^{n}_r}$. So, if the semiline $\vec{0z}$ we have $(t_{r_j}z_j,+\infty):=\vec{0z}\setminus (\vec{0z}\cap B^{n}_{t_{r_j}})\subset \Omega$, we have $\alpha(z)<+\infty$. Finally, note that $(z,+\infty)\subset \Omega\setminus B_r^{n}$.

\end{proof}

Moreover, for $r>R>0$, we have

\begin{eqnarray}\label{des.denconezr}
\nonumber	\liminf_{r\rightarrow +\infty}{\frac{\mathcal{H}^{n}(\Omega\cap B^{n}_r)}{r^{n}}} & = &  \liminf_{r\rightarrow +\infty}{\frac{\mathcal{H}^{n}((\Omega\setminus B^{n}_R)\cap B^{n}_r)}{r^{n}}}\\
	\nonumber & \geq & \liminf_{r\rightarrow +\infty}{\frac{\mathcal{H}^{n}((Z_r\setminus  B^{n}_R)\cap B^{n}_r)}{r^{n}}}\\
	\nonumber & = & \liminf_{r\rightarrow +\infty}{\frac{\mathcal{H}^{n}(Z_r\cap B^{n}_r)-\mathcal{H}^{n}(Z_r\cap B^{n}_R)}{r^{n}}}\\
	\nonumber & = & \liminf_{r\rightarrow +\infty}{\left(\mathcal{H}^{n}(Z_r\cap B^{n}_1)-\frac{\mathcal{H}^{n}(Z\cap B^{n}_R)}{r^{n}}\right)}\\
	\nonumber	& = &  \mathcal{H}^{n}(Z\cap B^{n}_1)=\mathcal{H}^{n}(B^{n}_1\cap C(\Omega,\infty)).
\end{eqnarray}
Therefore, $\Omega$ has density at infinity and 
$$
\theta^n(\Omega)=\frac{\mathcal{H}^{n}(B^{n}_1\cap C(\Omega,\infty))}{\mu_n}=\theta^n(C(\Omega,\infty)).
$$
\end{proof}

Now we prove that $X$ has density at infinity. Recently, this was proved in \cite{NguyenS:2023}, but here we present a different proof. 

\begin{claim}\label{propositiondensityexists}
$X$ has density at infinity.
\end{claim}
\begin{proof}[Proof of the Claim \ref{propositiondensityexists}]
Since the tangent mapping $\nu\colon \mathcal{M}:=X\setminus {\rm Sing}_1(X)\to G_n(\mathbb{R}^{n+k})$, given by $\nu(x)=T_x\mathcal{M}$, is a definable mapping (see \cite[Lemma 1.14]{KurdykaP:2006}),  likewise it was done in \cite{Stasica:1982}, we obtain that for each $\epsilon > 0$, there exist $N(\epsilon)$ and disjoint definable sets $\Gamma_1^{\epsilon}, \ldots, \Gamma_{N(\epsilon)}^{\epsilon}$ contained in $X$ such that:
\begin{enumerate}
	\item $\dim X\setminus \cup_{i=1}^{N(\epsilon)} \Gamma_i^{\epsilon} < n$;
	\item Each $\Gamma_i^{\epsilon}$ is the graph of a mapping $\varphi_i^{\epsilon}\colon U_i^{\epsilon}\rightarrow (E_i^{\epsilon})^{\bot}$ with derivative bounded by $\epsilon$ and $U_i^{\epsilon}$ is a definable open set in $E_i^{\epsilon}\in G_n(\mathbb{R}^{n+k})$, where $G_n(\mathbb{R}^{n+k})$ is the Grassmannian of all $n$-dimensional linear subspaces of $\mathbb{R}^{n+k}$.
\end{enumerate}

By refining the above decomposition, we can assume, if necessary that each $U_i^{\epsilon}$ is $M$-LNE, where $M=M(n)$ depends only on $n$ (see, e.g., \cite[Theorem 1.2]{KurdykaP:2006}). Therefore, each $\varphi_i^{\epsilon}$ is $M\epsilon$-Lipschitz.
Thus, the mapping  $\psi_i^{\epsilon}\colon U_i^{\epsilon}\to \Gamma_i^{\epsilon}$, given by $\psi_i^{\epsilon}(x)=x+\varphi_i^{\epsilon}(x)$ is a definable lipeomorphism such that
$$
\|x-y\|\leq \|\psi_i^{\epsilon}(x)-\psi_i^{\epsilon}(y)\|\leq (1+M\epsilon)\|x-y\|\quad \forall x,y\in U_i^{\epsilon}.
$$

By Claim \ref{propdensitycone}, it follows that each $U_i^{\epsilon}$ has a density at infinity and 
$$
\theta^n(U_i^{\epsilon})=\lim\limits_{r\rightarrow +\infty}{\frac{\mathcal{H}^n(U_i^{\epsilon}\cap B_r^{n})}{\mu_nr^n}}=\frac{\mathcal{H}^n(B_1^{n}\cap C(U_i^{\epsilon},\infty))}{\mu_n}=\theta^n(C(U_i^{\epsilon},\infty)).$$

Thus,
$$
\frac{\theta^n(U_i^{\epsilon})}{(1+M\epsilon)^n}\leq \liminf_{r\rightarrow +\infty}{\frac{\mathcal{H}^n(\Gamma_i^{\epsilon}\cap B_r^n)}{\mu_nr^n}}\leq \limsup_{r\rightarrow +\infty}{\frac{\mathcal{H}^n(\Gamma_i^{\epsilon}\cap B_r^n)}{\mu_nr^n}}\leq (1+M\epsilon)^n\theta^n(U_i^{\epsilon}).
$$

By setting $\lambda(\epsilon)=\sum_{i=1}^{N(\epsilon)}\theta^n(U_i^{\epsilon})$ and since for every $r > 0$, we have
$$\mathcal{H}^{n}(X\cap B_r^{n+k})=\sum_{i=1}^{N(\epsilon)}{\mathcal{H}^{n}(\Gamma_i^{\epsilon}\cap B_r^{n+k})},$$
then we obtain
$$\frac{\lambda(\epsilon)}{(1+M\epsilon)^n}\leq \liminf_{r\rightarrow +\infty}{\frac{\mathcal{H}^n(X\cap B_r^n)}{\mu_nr^n}}\leq \limsup_{r\rightarrow +\infty}{\frac{\mathcal{H}^n(X\cap B_r^n)}{\mu_nr^n}}\leq (1+M\epsilon')^n\lambda(\epsilon'),$$
for all $\epsilon, \epsilon'>0$.

Note that $\lambda(\epsilon)\leq (1+M\epsilon)^n2^n\lambda(1)$. Then
$$
\lim_{\epsilon\rightarrow 0}{(1+M\epsilon)^n\lambda(\epsilon)-\frac{\lambda(\epsilon)}{(1+M\epsilon)^n}}=0.$$

Therefore $\liminf_{r\rightarrow +\infty}{\frac{\mathcal{H}^n(X\cap B_r^n)}{\mu_nr^n}}= \limsup_{r\rightarrow +\infty}{\frac{\mathcal{H}^n(X\cap B_r^n)}{\mu_nr^n}}$, and thus $\theta^n(X)$ is defined.

\end{proof}
In this case, it is easy to see that
$$
\theta^n(X)=\lim \limits _{r\to +\infty }\frac{\mathcal{H}^n(X\cap B_r^{n+k}(x))}{\mu_n r^n}$$
for all $x\in \R^{n+k}$.

Finally, we are going to prove the second part, more precisely,
$$\theta^n(X)=\sum\limits _{j=1}^mk_{X,\infty}(C_j)\cdot \mathcal{H}^{n}(C_j),$$
where $C_1,...,C_m$ are the connected  components of $C_{{\rm Smp}}(X,\infty)$.

It is worth noting that the above formula is not presented in \cite{NguyenS:2023}.

Fixed $\epsilon>0$ and $i\in \{1,...,N(\epsilon)\}$, let $U=U_i^{\epsilon}$, $\Gamma=\Gamma_i^{\epsilon}$, $\varphi=\varphi_i^{\epsilon}$ and $E=E_i^{\epsilon}$. Therefore, it follows from Theorems \ref{inv_cones} and \ref{multiplicities} that

	\begin{enumerate}
		\item $C(\Gamma,\infty)$ is the graph of a definable $M\epsilon$-Lipschitz mapping $d_{\infty}\varphi\colon C(U,\infty)\to E^{\perp}$;
		\item $k_{\Gamma,\infty}(x)=1$ for all $x\in C_{\rm Smp}(\Gamma,\infty)$.
	\end{enumerate}

Then

$$\frac{1}{(1+M\epsilon)^{n}}\theta^n(U_i^{\epsilon})\leq \theta^n\left(\Gamma_{i}^{\epsilon}\right)\leq (1+M\epsilon)^n\theta^n(U_i^{\epsilon})$$
 and
 
$$\frac{1}{(1+M\epsilon)^{n}}\theta^n(C(U_i^{\epsilon},\infty))\leq \theta^n\left(C(\Gamma_{i}^{\epsilon},\infty)\right)\leq (1+M\epsilon)^n\theta^n(C(U_i^{\epsilon},\infty)).$$

Again, according to Claim \ref{propdensitycone}, we have
$$\theta^n(C_{\rm Smp}(U_i^{\epsilon},\infty))=\theta^n(C(U_i^{\epsilon},\infty))=\theta^n\left(U_i^{\epsilon}\right).$$
Therefore, 
\begin{equation}\label{eqdensity}
\frac{1}{(1+M\epsilon)^{2n}}\theta^n(C(\Gamma_i^{\epsilon},\infty))\leq \theta^n\left(\Gamma_{i}^{\epsilon}\right) \leq  (1+M\epsilon)^{2n}\theta^n(C(\Gamma_i^{\epsilon},\infty)).
\end{equation}

So, let $\mathcal{A}$ be a stratification of $C_{\rm Smp}(X,\infty)$ compatible with $\mathbb{S}^{{n+k}-1}$, $C_1,\cdots, C_m, C_{\rm Smp}(\Gamma_1^{\epsilon},\infty),\cdots, C_{\rm Smp}(\Gamma_{N(\epsilon)}^{\epsilon},\infty)$.

For all $1\leq i\leq N(\epsilon)$, $k_{\Gamma_{i}^{\epsilon},\infty}(x)=1$  for all $x\in C_{\rm Smp}(\Gamma_i^{\epsilon},\infty)$. Therefore, if $T\in \mathcal{A}$ is a stratum of dimension $n-1$ contained in $C_j\cap \mathbb{S}^{{n+k}-1}$, it follows from the definition of $k_{X,\infty}(C_j)$ that it is the number of pieces $\Gamma_{i}^{\epsilon}$ such that $T$ is contained in $C_{\rm Smp}(\Gamma_{i}^{\epsilon})$. Then
$$
	\sum_{i=1}^{N(\epsilon)}{\mathcal{H}^{n}(C_{\rm Smp}(\Gamma_i^{\epsilon})\cap B^{{n+k}})} = \sum_{j=1}^{m}{k_{X,\infty}(C_j)}\mathcal{H}^{n-1}(C_j\cap B^{{n+k}})
$$
and thus
\begin{eqnarray}\label{eqmulrelative}
	\sum_{i=1}^{N(\epsilon)}{\theta^{n}(C_{\rm Smp}(\Gamma_i^{\epsilon}))} & = & \sum_{j=1}^{m}{k_{X,\infty}(C_j)}\theta^{n}(C_j).
\end{eqnarray}
Therefore, from Eq. \eqref{eqdensity}, \eqref{eqmulrelative} and 
$$\theta^n(X)=\sum_{i=1}^{N(\epsilon)}{\theta^n(\Gamma_{i}^{\epsilon})},$$
we obtain 
$$
\theta^n(X)=\sum\limits _{j=1}^mk_{X,\infty}(C_j)\cdot \theta^n(C_j)=\sum\limits _{j=1}^mk_{X,\infty}(C_j)\cdot \frac{\mathcal{H}^{n}(C_j\cap B^{{n+k}})}{\mu_n}.
$$
\end{proof}

\subsection{Some examples}

\begin{example}
	Let $X=\left\{(x,y)\in \mathbb{R}^2;y-x^2=0\right\}$. Then, $\theta^1(X)=1$. 
\end{example}

\begin{example}
	Let $C=\{(x,y,z)\in \mathbb{R}^3; x^2+y^2=\cosh^2(z)\}$ 
	be the catenoid surface. Then, $\theta^2(C)=2$.
\end{example}
\begin{figure}[!htb]
	\centering
	\includegraphics[scale=0.2]{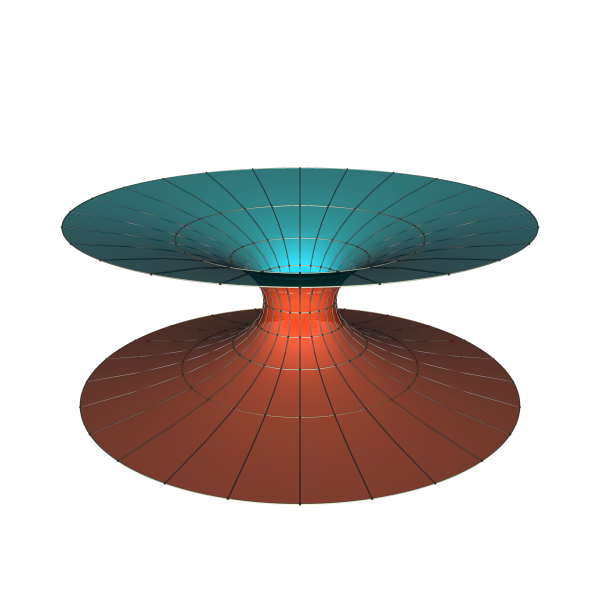}
	\caption{The catenoid minimal surface (see \cite{site_minimal})}
\end{figure}


\begin{example}
	Let $X=\{(x,y,z)\in \mathbb{R}^3; z^2=\alpha(x^2+y^2)\}$, where $\alpha> 0$.
	Then, $\theta^2(X)=2(1+\alpha)^{-\frac{1}{2}}$ (see \cite[Example 2.5]{Kurdyka:1989}). 
\end{example}

If we remove the hypothesis that $X$ is a definable set, then density can be infinite or cannot even exist.

\begin{example}
Let  $H=\left\{(x,y,z)\in \mathbb{R}^3;z=\tan^{-1}\left(\frac{y}{x}\right)\right\}$
be the helicoid surface. Then, $\theta^2(H)=+\infty$.
\end{example}
	\begin{figure}[!htb]
	\centering
	\includegraphics[scale=0.2]{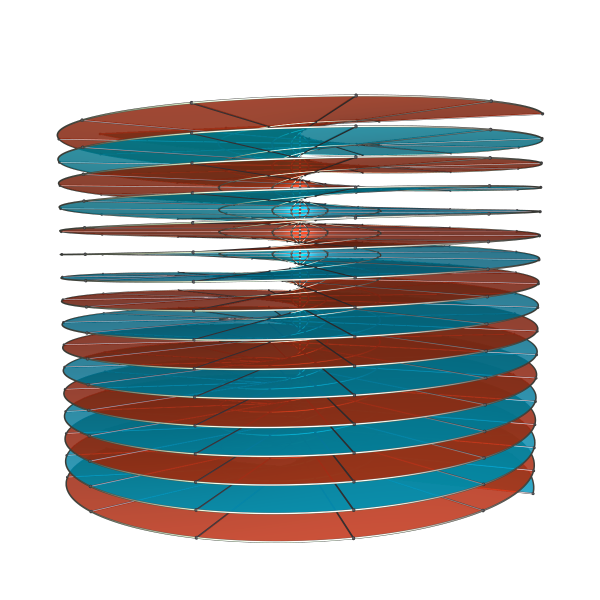}
	\caption {The helicoid minimal surface (see \cite{site_minimal})}
\end{figure}

\begin{example}
Let $I_1 = (0, a_1)$ be an open interval where $a_1 > 0$. Through recurrence, we define the following sets: $A_1:=I_1\times \{0\}$, $A_2:=I_2\times \left(\{-\frac{1}{2}\}\cup \{\frac{1}{2}\}\right)$, where $I_2=(a_1,a_1+2\mathcal{H}^1(I_1))$, $A_3:=I_3\times \{0\}$, where $I_3=(a_2,a_2+2\mathcal{H}^1(I_2))$ and in several cases 
	
	$$A_{2j+1}=I_{2j}\times \{0\},$$
	and 
	$$A_{2j}=I_{2j-1}\times \left(\left\{-\frac{1}{2}\right\}\cup \left\{\frac{1}{2}\right\}\right),$$
where and  $I_j=(a_{j-1},a_{j-1}+2\mathcal{H}^1(I_{j-1})),$ and $\mathcal{H}^1(I_{k})$ is the length of the interval $I_k$.
Thus we define 	
	$$X=\bigcup_{j\in \mathbb{N}}{A_j}.$$
	Then, the following limit does not exist 
	$$	 
	\lim_{r\rightarrow +\infty}\frac{\mathcal{H}^1(X\cap B_r^{2})}{\mu_1 r}.
	$$
\end{example}

\subsection{First consequences}

\subsubsection{Degree of an algebraic set is its density at infinity}
\begin{corollary}\label{cor:density_degree}
 Let $X\subset \C^{n+k}$ be a pure-dimensional complex algebraic set with $\dim_{\C}X=n$. Then, $\theta^{2n}(X)={\rm deg}(X).$
\end{corollary}
\begin{proof}
Note that when $X$ is a complex algebraic set, there is a complex analytic set $\sigma$ with $\dim \sigma <\dim X$, such that for each irreducible component $X_j$ of the tangent cone $C(X,\infty)$, $X_j\setminus \sigma$ intersect only one connected component $C_i$ of $C_{\rm Smp}(X,\infty)$, then we define $k_{X,\infty}(X_j):=k_{X,\infty}(C_i)$. Thus, from \cite[\S 2.3]{FernandezFS:2018}, we obtain
\begin{equation*}\label{kurdyka-raby}
{\rm deg}(X)=\sum\limits_{j=0}^r k_{X,\infty }(X_j)\cdot{\rm deg}(X_j),
\end{equation*}
where $X_1,...,X_r$ are all the irreducible components of $C(X,\infty)$. Since each $X_j$ is a complex algebraic cone, we have that ${\rm deg}(X_j)=m(X_j,0)$, where $m(X_j,0)$ denotes the multiplicity of $X_j$ at $0$ (see the definitions of multiplicity and degree in \cite{Chirka:1989}). 
By \cite[Theorem 7.3]{Draper:1969}, $\theta^{2n}(X_j,0)=m(X_j,0)$ for each $j\in\{1,...,r\}$. Since each $X_j$ is a cone with vertex at $0$, we have that $\theta^{2n}(X_j,0)=\theta^{2n}(X_j)$. Thus,
$$
{\rm deg}(X)=\sum\limits_{j=0}^rk_{X,\infty }(X_j)\cdot \theta^{2n}(X_j).
$$
By Theorem \ref{thm:densitymult}, we have that 
$$
\theta^{2n}(X)=\sum\limits_{j=0}^r k_{X,\infty }(X_j)\cdot \theta^{2n}(X_j).
$$
Therefore $\theta^{2n}(X)={\rm deg}(X)$.
\end{proof}

\subsubsection{The o-minimal Chow's theorem}
Let us remember the Stoll-Bishop's Theorem:
\begin{theorem}[Stoll-Bishop's Theorem (see \cite{Stoll:1964a,Stoll:1964b,Bishop:1964})]\label{Stoll-Bishop}
Let $Z\subset\C^m$ be a pure-dimensional entire complex analytic subset with $\dim_{\C}X=n$. Then $Z$ is a complex algebraic set if and only if there exists a constant $R>0$ such that 
$$
\frac{\mathcal{H}^{2n}(Z\cap \overline{B}_r^{2m}(0))}{r^{2n}}\leq R, \quad \mbox{ for all } r>0.
$$
\end{theorem}
Thus, we obtain the o-minimal Chow's Theorem proved in \cite[Corollary 4.5]{PeterzilS:2009}:
\begin{corollary}[O-minimal Chow's theorem]\label{o-minimal_chow_thm}
 Let $X\subset \C^{n+k}$ be a pure-dimensional entire complex analytic set with $\dim_{\C}X=n$. Then, $X$ is a complex algebraic set if and only if $X$ is definable in an o-minimal structure on $\R$.
\end{corollary}
\begin{proof}
 By Theorem \ref{thm:densitymult}, $\theta^{2n}(X)$ is finite. Then there exists a constant $R>0$ such that 
$$
\frac{\mathcal{H}^{2n}(X\cap \overline{B}_r^{2(n+k)}(0))}{r^{2n}}\leq R, \quad \mbox{ for all } r>0.
$$

Therefore, by Stoll-Bishop's Theorem, $X$ is a complex algebraic set.
\end{proof}

\subsubsection{The Lipschitz o-minimal Chow's theorem}
\begin{corollary}[Lipschitz o-minimal Chow's theorem]\label{lip_o-minimal_chow_thm}
 Let $X\subset \C^{n+k}$ be a pure-dimensional entire complex analytic set with $\dim_{\C}X=n$. Then, $X$ is a complex algebraic set if and only if $X$ is lipeomorphic at infinity to a definable set in an o-minimal structure on $\R$.
\end{corollary}
\begin{proof}
We closely follow the proof of Theorem 3.1 in \cite{Sampaio:2023}.

We assume that $0\in X$ and that $X$ is lipeomorphic at infinity to a definable set in an o-minimal structure on $\R$. Let $A\subset \R^m$ be a definable set such that there exist compact subsets $K\subset \C^{n+k}$ and $\tilde K\subset \R^m$ and a lipeomorphism $\varphi\colon A\setminus \tilde K\to X\setminus  K$. Let $\lambda\geq 1$ such that
$$
\frac{1}{\lambda}\|x-y\|\leq \|\varphi(x)-\varphi(y)\|\leq \lambda \|x-y\|, \quad  \forall x,y\in A\setminus \tilde K.
$$
Fix $x_0\in A\setminus \tilde K$ and set $y_0=\varphi(x_0)$. Let $\tilde r_0=\|x_0\|$ and $r_0=\|y_0\|$. 
Thus, for any $r>0$, we obtain that 
$$
(X\setminus K)\cap B_{r}^{2(n+k)}(0)\subset \varphi ((A\setminus \tilde K)\cap B_{\lambda(r+r_0)+\tilde r_0}^{m}(0))
$$
and 
\begin{eqnarray*}
\mathcal{H}^{2n}(X\cap B_{r}^{2(n+k)}(0))&\leq &\lambda^{2n}\mathcal{H}^{2n}(A\cap B_{\lambda(r+r_0)+\tilde r_0}^{m}(0))+ \mathcal{H}^{2n}(X\cap K).
\end{eqnarray*}

Since $A$ is a definable set, by Theorem \ref{thm:densitymult}, $\theta^{2n}(A)$ is finite. Then there exists a constant $C>0$ such that 
$$
\frac{\mathcal{H}^{2n}(A\cap \overline{B}_r^{m}(0))}{r^{2n}}\leq C, \quad \mbox{ for all } r>0.
$$

But $X$ is analytic at $0$, then there exist $r_1,R_1>0$ such that 
$$
\frac{\mathcal{H}^{2n}(X\cap B_{r}^{2(n+k)}(0))}{r^{2n}} \leq R_1,
$$
for all $r\leq r_1$. Moreover, $\mathcal{H}^{2n}(X\cap K)<+\infty$. Then, there exists $R_2>0$ such that
$$
\frac{\mathcal{H}^{2n}(X\cap B_{r}^{2(n+k)}(0))}{r^{2n}}\leq \frac{\lambda^{2n}C(\lambda(r+r_0)+\tilde r_0)^{2n}+\mathcal{H}^{2n}(X\cap K)}{r^{2n}}\leq R_2,
$$
for all $r\geq r_1$.

Therefore, by Stoll-Bishop's Theorem, $X$ is a complex algebraic set.
\end{proof}

\subsubsection{O-minimal Yau's Bernstein Problem}
\begin{corollary}\label{o-minimal_yau_conj}
 Let $X\subset \R^{n+1}$ be an oriented stable complete minimal hypersurface with $n\leq 6$. If $X$ is a definable set then $X$ is a hyperplane.
\end{corollary}
\begin{proof}
By Theorem \ref{thm:densitymult}, we have
$$
\lim\limits_{r\to +\infty}\frac{\mathcal{H}^{n}(X\cap B_{r}^{n+1}(0))}{r^{n}}<+\infty.
$$
Thus, by \cite[Corollary, p. 104]{ShenZ:1998} (see also \cite{ShoenS:1981} and \cite{ShoenSY:1975}), $X$ is a hyperplane.
\end{proof}

\section{Parametric versions of the Bernstein Theorem}\label{sec:parametric_version}

\begin{theorem}\label{main_thm}
Let $X\subset \mathbb R^{n+k}$ be a closed and connected set with $n=\dim_H X$. Assume that $X$  satisfies the monotonicity formula at some $p\in X$. Then, the following statements are equivalent.
\begin{itemize}
 \item [(1)] $X$ is an affine linear subspace;
 \item [(2)] $X$ is a definable set, blow-spherical regular at infinity and $C(X,\infty)$ is a linear subspace;
 \item [(3)] $X$ is a definable set that is Lipschitz regular at infinity and $C(X,\infty)$ is a linear subspace;
 \item [(4)] $X$ is an LNE at infinity definable set and $C(X,\infty)$ is a linear subspace;
 \item [(5)] $\theta^n(X)=1$.
\end{itemize}
\end{theorem}
\begin{proof}
It is obvious that item $(1)$ implies the items $(2)$, $(3)$, $(4)$ and $(5)$.

\bigskip

\noindent $(2)\Longrightarrow (5)$. Assume that $X$ is a definable set, blow-spherical regular at infinity and that $C(X,\infty)$ is a linear subspace.

Since $X$ is blow-spherical regular at infinity, by Proposition 2.14 in \cite{SampaioS:2023}, we have $k_{X,\infty}(C)=1$ for all connected component $C$ of $C_{{\rm Smp}}(X,\infty)$. It follows from the hypothesis that $C(X,\infty)$ is a hyperplane and Theorem \ref{thm:densitymult} that $\theta^n(X)=1$.

\bigskip

\noindent $(3)\Longrightarrow (5)$. Assume that $X$ is a definable set that is Lipschitz regular at infinity and $C(X,\infty)$ is a linear subspace.

By Theorem \ref{multiplicities}, we obtain that $k_{X,\infty}(v)=1$, for all $v\in C_{\rm Smp}(X,\infty)$. 
Therefore, as $C(X,\infty)$ is a linear subspace, it follows from Claim \ref{claimrelmul}, Remark \ref{remarkdense}, and Theorem \ref{thm:densitymult} that $\theta^{n}(X)=\mathcal{H}^{n}(C_1)=1$.

\bigskip

\noindent $(4)\Longrightarrow (5)$. Assume that $X$ is a definable set that is LNE at infinity and $C(X,\infty)$ is a linear subspace.
	\begin{claim}\label{claimrelmul}
		If $X\subset \R^m$ is a definable set and LNE at infinity, then $k_{X,\infty}(v)=1$ for all $v\in C_{\rm Smp}(X,{\infty})$.
	\end{claim}
	\begin{proof}[Proof of Claim \ref{claimrelmul}]
Indeed, we suppose there exists  $x=(x',0)\in Smp(\partial X_{\infty }')=C(X,\infty)\cap \mathbb{S}^{m-1}\times \{0\}$ such that $k_{X,\infty}(x')\geq 2$, we know that $k_{X,\infty}(X_j)=k_{X,\infty}(x')$ is the number of connected components of the germ $(\rho_{\infty }^{-1}(X),x)$, where $X_j$ is a connected component of $C_{\rm Smp}(X,{\infty})$ and $x'\in X_j$. Thus, for a sequence of positive real numbers $\{t_j\}_{j\in \mathbb{N}}$ such that $\lim t_j=+\infty$,  $k_{X,\infty}(x')$ can be seen as the number of connected components of the set 
$(SX\cap \mathbb{S}^{m-1}\times \{t_j^{-1}\})\cap B_{\delta}^{m+1}(x)$ for all large $k$ and a small enough $\delta>0$, where 
 $$
 SX=\{(x,s)\in \mathbb{S}^{m-1}\times (0,+\infty);\,\textstyle\frac{1}{s}\cdot x \in X \mbox{ and }\frac{1}{s}\in S\}
 $$
and $S=\{t_j^{-1}\}_{j\in \mathbb{N}}$. 
 
Let $X_1$ and $X_2$ be two connected components of $X_{\infty}'\cap B_{\delta}^{m+1}(x),$ and for $\eta, R>0$, we define the following set $C_{\eta,R}(x')=\{v\in \mathbb{R}^n\setminus \{0\}; \exists t>0;\|v-tx'\|\leq \eta t\}\setminus B_R^n$.
For each $j\in \mathbb{N}$, we can take $x_j\in \rho_{\infty }(X_1)$ and $y_j\in \rho_{\infty }(X_2)$ such that $\|x_j\|=\|y_j\|=t_j$ and $\lim_{j\rightarrow +\infty}{\frac{x_j}{t_j}}=\lim_{j\rightarrow +\infty}{\frac{y_j}{t_j}}=x'$. Now, taking subsequence, if necessary, we assume that $x_j,y_j\in C_{\frac{\eta}{2},R}$. Thus, let $\beta_j\colon [0,1]\rightarrow X$ be a curve connecting $x_j$ to $y_j$, we have that there exists a $t_0\in [0,1]$ such that $\beta_j(t_0)\notin C_{\eta,R}$, since $x_j$ and $y_j$ belong to different connected components of $X\cap C_{\eta,R}$. Therefore, $lenght(\beta_j)\geq \eta t_j$ and then $d_X(x_j,y_j)\geq \frac{\eta}{2} t_j$. 

On the other hand, $X$ is a Lipschitz normally embedded at infinity set, then there exists a ball $B_r^n$ and a constant $C > 0$ such that $d_{X\setminus B_r^n}(v, w)\leq  C\|v-w\|$, for all $v, w\in X\setminus B_r(0)$. Therefore, $C\|\frac{x_j}{t_j}-\frac{y_j}{t_j}\|\geq \eta$, for all $j\in \mathbb{N}$ and this is a contradiction, since $\lim{\frac{x_j}{t_j}}= \lim{\frac{y_j}{t_j}}$.

Therefore $k_{X,\infty}(v)=1$ for all $v\in C_{\rm Smp}(X,{\infty})$.
\end{proof}

Therefore, as $C(X,\infty)$ is a linear subspace, it follows from Claim \ref{claimrelmul}, Remark \ref{remarkdense}, and Theorem \ref{thm:densitymult} that $\theta^{n}(X)=\mathcal{H}^{n}(C_1)=1$.

\bigskip

\noindent $(5)\Longrightarrow (1)$. Suppose that $\theta^n(X)=1$. 

Since $X$ satisfies the monotonicity formula at $p$, $\theta^n(X,p,\cdot)$ is a nondecreasing function such that $1\leq \theta^n(X,p,r)\leq \theta^n(X)=1$ for all $r>0$. Then $\theta^n(X,p,r)=1$ for all $r>0$, and by using again that $X$ satisfies the monotonicity formula at $p$, we have that $X$ is a cone with vertex at $p$ and $X$ is $C^1$-smooth at $p$. Therefore $X$ is an affine linear subspace.
\end{proof}

\subsection{Some examples}
\begin{example}\label{example:catenoid}
 The catenoid is the surface $C=\{(x,y,z)\in\R^3; x^2+y^2=\cosh^2(z)\}$.  The Catenoid is a complete embedded minimal hypersurface and its geometric tangent cone at infinity is the plane $z=0$. Moreover, it is definable in the o-minimal structure $\R_{exp}$ (see \cite{Dries:1984}).
\end{example}
\begin{example}
Let $L\subset \mathbb{S}^2$ be a curve lipeomorphic to $\mathbb{S}^1$ with $2\pi<\mathcal{H}^1(L)< +\infty$. 
 Thus, the surface $X=Cone_0(L):=\{tx;x\in L$ and $t\geq 0\}$ is lipeomorphic to $\R^2$ and, in particular, it is Lipschitz regular at infinity. Since $X$ is a cone, then $\theta^2(X,0,r)=\mathcal{H}^1(L)/(2\pi)>1$ for all $r>0$. Therefore $X$ satisfies the monotonicity formula at $0$. However, it is not a plane. 
\end{example}
\begin{example}
 The surface $P=\{(z,w)\in\C^2; z=w^2\}$ is a smooth complex algebraic curve and, in particular, it is an area-minimizing set. Its geometric tangent cone at infinity is the complex line $w=0$. However, $P$ is not a plane. 
\end{example}

\begin{example}
We cannot remove the hypothesis that $X$ satisfies the monotonicity formula at some $p$. Indeed, $X=\{(x,y,z)\in\R^3; z^3=x^{2}+y^{2}+1\}$ is lipeomorphic to $\R^2$ and, in particular, it is Lipschitz regular at infinity. In fact, $X$ is the graph of the Lipschitz function $f\colon \R^2\to \R$ given by $f(x,y)=(x^2+y^2+1)^{\frac{1}{3}}$. Moreover, $X$ is a smooth surface, $C(X,\infty)$ is the plane $z=0$ and $\theta^2(X)=1$. However, $X$ is not a plane. 
\end{example}

\subsection{Some direct consequences}

It is a direct consequence of Theorem \ref{thm:densitymult} and Claim \ref{claimrelmul} the following result:
		
\begin{corollary}\label{cor:lne_density}
Let $X\subset \R^{n+k}$ be an $n$-dimensional definable set. If $X$ is LNE at infinity, then $\theta^n(X)=\theta^n(C(X,\infty))$.
\end{corollary}
In particular, since for a pure dimensional complex algebraic set, $X\subset \C^{n+k}$ with $\dim_{\C}X=n$, we have $\theta^{2n}(X)={\rm deg}(X)$ (see Corollary \ref{cor:density_degree}), we obtain the main result of \cite{DiasR:2022}:
\begin{corollary}
Let $X\subset \C^{n+k}$ be pure dimensional complex algebraic set. If $X$ is LNE at infinity, then $\deg(X)=\deg(C(X,\infty))$.
\end{corollary}

We also obtain the following:
\begin{corollary}
Let $X\subset \mathbb R^{n+k}$ be a pure-dimensional algebraic set. Suppose that $X$ is LNE at infinity and $C(X,\infty)$ is a linear subspace. If  $X$ is a minimal submanifold or an area-minimizing set, then $X$ is an affine linear subspace.
\end{corollary}
\begin{proof}
By Corollary \ref{cor:lne_density}, $\theta^n(X)=\theta^n(C(X,\infty))$. Since $C(X,\infty)$ is a hyperplane, we have $\theta^n(X)=1$. By Theorem \ref{main_thm}, $X$ is an affine linear subspace.
\end{proof}
We obtain the main results of \cite{FernandesS:2020} (see also \cite{Sampaio:2023}):
\begin{corollary}
Let $X\subset \mathbb C^{n+k}$ be a pure-dimensional complex analytic set with $\dim_{\C}X=n$. Suppose that $X$ is LNE at infinity and $C(X,\infty)$ is a dimensional complex linear subspace with $\dim_{\C}C(X,\infty)=n$. Then $X$ is an affine linear subspace.
\end{corollary}
\begin{proof}
Since $C(X,\infty)$ is a dimensional complex algebraic set with $\dim_{\C}C(X,\infty)=n$, we have that $X$ is a complex algebraic set as well (see, e.g., \cite[Theorem 3.1]{Sampaio:2023}). In particular, $X$ is a definable set and satisfies the monotonicity formula at any $p\in X$. By Theorem \ref{thm:densitymult}, $X$ is an affine linear subspace.
\end{proof}

\begin{corollary}
Let $X\subset \mathbb C^{n+k}$ be a pure-dimensional complex algebraic set with $\dim_{\C}X=n$. Suppose that $X$ is Lipschitz regular at infinity. Then $X$ is an affine linear subspace.
\end{corollary}
\begin{proof}
$X$ is Lipschitz regular at infinity, then $X$ is LNE at infinity, and by Theorem \ref{inv_cones}, $C(X,\infty)$ is lipeomorphic to $\R^{2n}$. In particular, $C(X,\infty)$ is a topological manifold, and thus by Prill's theorem (\cite[Theorem]{Prill:1967}), $C(X,\infty)$ is a dimensional complex linear subspace with $\dim_{\C}C(X,\infty)=n$. By Theorem \ref{thm:densitymult}, $X$ is an affine linear subspace. 
\end{proof}

We obtain also the following version of Moser's result \cite[Corollary, p. 591]{Moser:1961}:
\begin{corollary}\label{cor:moser_thm_def}
Let $X\subset \mathbb R^{n+1}$ be a complete minimal hypersurface. Suppose that there are compact sets $K\subset \R^n$ and $\tilde K\subset \R^{n+1}$ such that $X\setminus \tilde K$ is the graph of a definable Lipschitz function $u\colon \mathbb{R}^{n}\setminus K\to \mathbb{R}$. Then $u$ is the restriction of an affine function and, in particular, $X$ is a hyperplane.
\end{corollary}
\begin{proof}
We need a preliminary result. For a set $A\subset \R^n$, denote by $\mathcal{N}(A,\infty)$ the union of all hyperplanes $T$ such that there is a sequence $\{x_j\}_j\subset A\setminus {\rm Sing}_1(A)$ such that $\lim\|x_j\|=+\infty$ and $T_{x_j}A$ converges to $T$.
\begin{lemma}
Let $A\subset \R^n$ be an unbounded definable set. Then, $C(A,\infty)\subset \mathcal{N}(A,\infty)$. 
\end{lemma}
\begin{proof}
Firstly, note that $C(A\setminus {\rm Sing}_1(A),\infty)=C(A,\infty)$ and $\mathcal{N}(A\setminus {\rm Sing}_1(x),\infty)=\mathcal{N}(A,\infty)$. Thus, we may assume that $A$ is $C^1$-smooth.

Let $v\in C(A,\infty)$. We are going to show that $v\in \mathcal{N}(A,\infty)$ and we may assume that $\|v\|=1$. In order to do that, it is enough to find,  for some $R>0$, a definable $C^1$-smooth arc $\alpha\colon (R,+\infty)\to A$ such that $\alpha(t)=tv+o_{\infty}(t)$. Indeed, by taking a subsequence, if necessary, we may assume that $T_{\alpha(j)} A$ converges to $T\subset \mathcal{N}(A,\infty)$. Since $\alpha'(j)\in T_{\alpha(j)} A$ for all $j$, then $\lim \alpha'(j)=v\in T$.

Now, let us prove that there is such an arc $\alpha$.
Let $Y=\rho_{\infty}^{-1}(A)$. We have that $Y$ is a definable set. Note that the closure of $Y$ in $\mathbb{S}^{n-1}\times [0,\infty)$ satisfies the following:
$$
\overline{Y}\cap (\mathbb{S}^{n-1}\times \{0\})=(C(A,\infty)\cap \mathbb{S}^{n-1})\times \{0\}.
$$
Thus, $(v,0)\in \overline{Y}$. By the Curve Selection Lemma, there is a definable continuous arc $\beta\colon [0,\epsilon)\to \overline{Y}$ such that $\beta((0,\epsilon))\subset Y$ and $\beta(0)=(v,0)$. Let us write $\beta(s)=(x(s),h(s))$. Since $h$ is not a constant function around $0$, by Monotonicity Lemma, one can suppose that $\beta$ is $C^1$-smooth and that $h$ is $C^1$ smooth, strictly increasing in the domain $(0,\delta)$ and $h'(s)\not=0$ for all $s\in (0,\delta)$. Hence, $h\colon (0,\delta/2)\to (0,\delta' )$ is a $C^1$-diffeomorphism, where $\delta' =h(\frac{\delta}{2})$.
Let $\alpha\colon (\frac{1}{\delta'},+\infty)\to A$ be the arc $\alpha(t)=\rho_{\infty}\circ \beta\circ h^{-1}(\frac{1}{t})$.

Clearly, $\alpha$ is $C^1$-smooth, and since $\lim\limits_{s\to 0^+} x(s)=x(0)=v$, we obtain
$$
\textstyle\alpha(t)=tx(h^{-1}(\frac{1}{t}))=tv+o_{\infty}(t).
$$
\end{proof}
From the above lemma, we obtain that $C(X,\infty)\subset \mathcal{N}(X,\infty)$.

By Lemma 1.18 in \cite{Simon:1987}, we have that the following limit exists: 
$$
\lim\limits_{x\to \infty}(D u(x),-1)=w.
$$

Therefore, $\mathcal{N}(X,\infty)=w^{\perp}$. In particular, $C(X,\infty)$ is a subset of an $n$-dimensional hyperplane, and thus $\mathcal{H}^{n}(C(X,\infty)\cap B_1^{n+1})\leq \mu_n$. 
Since there is a blow-spherical homeomorphism at infinity between $X$ and $\R^n$, then there is only one relative multiplicity which is $1$. Then, by Theorem \ref{thm:densitymult}, $\theta^n(X)\leq 1$. However, since $X$ is a minimal submanifold, $\theta^n(X)\geq 1.$ Therefore,  $\theta^n(X)= 1$ and by Theorem \ref{main_thm}, $X$ is a hyperplane. It is easy to verify that $u$ is the restriction of an affine function.
\end{proof}

In the next section, we generalize Corollary \ref{cor:moser_thm_def} (see Corollary \ref{thm:gen_moser_thm}).

\section{Generalization of the Moser's Bernstein Theorem}\label{sec:gen_bernstein_thm}
In this section, we generalize the Moser's Bernstein Theorem.

\begin{theorem}\label{thm:gen_moser_thm_higher}
Let $X\subset \mathbb R^{n+k}$ be a connected closed set with Hausdorff dimension $n$. Then, $X$ is an $n$-dimensional affine linear subspace if and only if we have the following: 
\begin{enumerate}
 \item $X$ satisfies the monotonicity formula at some $p\in X$; 
 \item there are compact sets $K\subset \R^n$ and $\tilde K\subset \R^{n+k}$ such that $X\setminus \tilde K$ is the graph of a $C^1$-smooth function $u\colon \mathbb{R}^{n}\setminus K\to \mathbb{R}^{k}$ with bounded derivative;
 \item $\mathcal{N}(X,\infty)$ is a linear subspace.
\end{enumerate} 
\end{theorem}
\begin{proof}
It is obvious that if $X$ is an $n$-dimensional affine linear subspace, then $X$ satisfies (1), (2) and (3).

Reciprocally, assume that $X$ satisfies (1), (2) and (3).

As before, for a set $A\subset \R^n$, denote by $\mathcal{N}(A,\infty)$ the union of all hyperplanes $T$ such that there is a sequence $\{x_j\}_j\subset A\setminus {\rm Sing}_1(A)$ such that $\lim\|x_j\|=+\infty$ and $T_{x_j}A$ converges to $T$.

Let $P$ be the hyperplane $\mathcal{N}(X,\infty)$. We choose linear coordinates $(y_1,...,y_{n+k})$ of $\R^{n+k}$ such that $P=\{(y_1,...,y_{n+k})\in\R^{n+k}; y_{n+1}=...=y_{n+k}=0\}$. For a larger enough $R>0$, $X\setminus (\overline{B_R^{n}}\times \R^k)$ is the graph of a function $v=(v_1,...,v_k)\colon P\setminus \overline{B_R^n}\to \R^k$. By the Implicit Mapping Theorem, we have that $v$ is a Lipschitz function and that the following limit exists: 
$$
\lim\limits_{y\to \infty}D v_i(y)=\tilde w_i,\quad i\in\{1,...,k\}.
$$
Therefore $\tilde w=(\tilde w_1,...,\tilde w_k)=0$.
Then, for each $\varepsilon>0$, there is $r>0$ such that $\|D v(y)\|\leq \varepsilon$ for all $y\in \Omega:= P\setminus \overline{B_r^n}$. 

\begin{claim}\label{claim:Lipschitz_function}
 $v|_{\Omega}$ is a Lipschitz mapping with constant $\pi \varepsilon$.
\end{claim}
\begin{proof}
Firstly, we define the set of all piecewise $C^1$-smooth curves connecting $x$ to $y$ in $\Omega$. We denote such a set by:
$$\Omega(x,y):=\left\{\gamma\colon[0,1]\rightarrow \Omega; \gamma(0)=x, \gamma(1)=y \mbox{ and } \gamma \mbox{ is piecewise $C^1$-smooth}\right\}.$$
Let $x, y \in \Omega$ be any two points. Then

\begin{eqnarray*}
	\|v(x)-v(y)\| & \leq &  \inf_{\gamma\in \Omega(x,y)}\int_{0}^{1}{\|(v\circ \gamma)'(t)\|dt}\\
	& = & \inf_{\gamma\in \Omega(x,y)}\int_{0}^{1}{\|D v(\gamma(t))\cdot\gamma'(t)\|dt}\\
	& \leq & \inf_{\gamma\in \Omega(x,y)}\int_{0}^{1}{\|\nabla(\gamma(t))\| \|\gamma'(t)\|dt}\\
	& \leq & \varepsilon\inf_{\gamma\in \Omega(x,y)}\int_{0}^{1}{\|\gamma'(t)\|dt}\\
	& = & \varepsilon d_{\Omega}(x,y).
\end{eqnarray*}
However, $\Omega$ is a set that is LNE such that $d_{\Omega}(x,y)\leq \pi\|x-y\|$ for all $x,y\in \Omega$. Therefore,
$$\|v(x)-v(y)\|\leq \pi \varepsilon \|x-y\|.$$
\end{proof}

\begin{claim}
$
1\leq \theta^n(X)\leq (1+\pi\varepsilon)^n.
$
\end{claim}
\begin{proof}
Let $\varphi\colon \Omega \rightarrow A=X\setminus (\overline{B_r^{n}}\times \R^k)$ be the mapping given by $\varphi(x)=(x,v(x))$. We have that $\varphi$ is a bi-Lipschitz mapping such that
$$
\|x-y\|\leq \|\varphi(x)-\varphi(y)\|\leq (1+\pi\varepsilon)\|x-y\|.
$$
Therefore,
$$
\mathcal{H}^n(B^n_{t}\setminus B^n_{r}) \leq \mathcal{H}^n(A\cap B_{t}^{n+k}) \leq (1+\pi\varepsilon)^n\mathcal{H}^n(B^n_{t+\pi\varepsilon}),
$$
for all $t>r$.
Thus, we obtain the following
$\theta^n(\R^n\setminus B^n_{r})\leq \theta^n(X)\leq (1+\pi\varepsilon)^n \theta^n(\R^n).$
Since $\theta^n(\R^n\setminus B^n_{r})=\theta^n(\R^n)=1$ and $\theta^n(A)=\theta^n(X)$, we obtain that $
1\leq \theta^n(X)\leq (1+\pi\varepsilon)^n
$, which finishes the proof of the claim.
\end{proof}

Setting $\varepsilon \to 0^+$, we obtain that $\theta^n(X)=1$. By Theorem \ref{main_thm}, $X$ is an affine linear subspace.

%

\end{proof}

\subsection{Examples}\label{subsec:examples_graph}
In this subsection, by presenting several examples, we show that Theorem \ref{thm:gen_moser_thm_higher} is sharp in the sense that no hypothesis of it can be removed.

We cannot remove the hypothesis that $\mathcal{N}(X,\infty)$ is a linear subspace.
\begin{example}[Theorem 7.1 in \cite{LawsonO:1977}]
There is a connected closed semialgebraic set $X\subset \mathbb R^{n+k}$ which is an area-minimizing set and is the graph of a semialgebraic Lipschitz function $u\colon \mathbb{R}^{n}\to \mathbb{R}^k$. Moreover, $X$ is not an affine linear subspace.
\end{example}

We cannot remove the condition that $X$ is the graph of a function outside of a compact.
\begin{example}
 The catenoid $C=\{(x,y,z)\in\R^3; x^2+y^2=\cosh^2(z)\}$ (see Example \ref{example:catenoid}) satisfies the monotonicity formula at any point $p\in C$ and $N(C, \infty)$ is the plane $z=0$; however, $C$ is not a plane.
\end{example}

We cannot remove the hypothesis that the function $u$ has bounded derivative.
\begin{example}
 The surface $P=\{(z,w)\in\C^2; z=w^2\}$ is a smooth complex algebraic curve and, in particular, it is an area-minimizing set. $N(P,\infty)$ is the complex line $w=0$. However, $P$ is not a plane. 
\end{example}

We cannot remove the hypothesis that $X$ satisfies the monotonicity formula at some $p$.
\begin{example} $X=\{(x,y,z)\in\R^3; z^3=x^{2}+y^{2}+1\}$ is the graph of the smooth Lipschitz function $f\colon \R^2\to \R$ given by $f(x,y)=(x^2+y^2+1)^{\frac{1}{3}}$. Moreover, $X$ is a smooth surface and $\mathcal{N}(X,\infty)$ is the plane $z=0$. However, $X$ is not a plane. 
\end{example}

\subsection{Consequences}

\begin{corollary}\label{thm:gen_moser_thm}
Let $X\subset \mathbb R^{n+1}$ be a connected closed set with Hausdorff dimension $n$ and that satisfies the monotonicity formula at some $p\in X$. Suppose that there are compact sets $K\subset \R^n$ and $\tilde K\subset \R^{n+1}$ such that $X\setminus \tilde K$ is a minimal hypersurface that is the graph of a $C^2$-smooth function $u\colon \mathbb{R}^{n}\setminus K\to \mathbb{R}$ with bounded derivative. Then $u$ is the restriction of an affine function and, in particular, $X$ is a hyperplane.
\end{corollary}
\begin{proof}
By Lemma 1.18 in \cite{Simon:1987} and Theorem X in \cite{Bers:1951}, we have that the following limit exists: 
$$
\lim\limits_{x\to \infty}(D u(x),-1)=w.
$$

Therefore, $\mathcal{N}(X,\infty)$ is the hyperplane $w^{\perp}$. 

By Theorem \ref{thm:gen_moser_thm_higher}, $X$ is a hyperplane. It is obvious that $u$ is the restriction of an affine function.
\end{proof}

\begin{corollary}\label{cor:gen_bernstein_thm}
Let $X\subset \mathbb R^{n+1}$ be a connected closed set with Hausdorff dimension $n$ and that satisfies the monotonicity formula at some $p\in X$. Suppose that $2\leq n\leq 7$ and there are compact sets $K\subset \R^n$ and $\tilde K\subset \R^{n+1}$ such that $X\setminus \tilde K$ is a minimal hypersurface that is the graph of a $C^2$-smooth function $u\colon \mathbb{R}^{n}\setminus K\to \mathbb{R}$. Then $u$ is the restriction of an affine function and, in particular, $X$ is a hyperplane.
\end{corollary}
\begin{proof}
By Theorem 1 in \cite{Simon:1987} and Theorem X in \cite{Bers:1951}, $D u$ is bounded. By Corollary \ref{thm:gen_moser_thm}, $X$ is a hyperplane and $u$ is the restriction of an affine function.
\end{proof}

Since a minimal hypersurface satisfies the monotonicity formula at any of its points, we obtain, as direct consequences of Corollaries \ref{cor:gen_bernstein_thm} and \ref{thm:gen_moser_thm}, the following results:

\begin{corollary}\label{cor:weak_gen_bernstein_thm}
Let $X\subset \mathbb R^{n+1}$ be a connected complete minimal hypersurface. Suppose that there are compact sets $K\subset \R^n$ and $\tilde K\subset \R^{n+1}$ such that $X\setminus \tilde K$ is the graph of a $C^2$-smooth function $u\colon \mathbb{R}^{n}\setminus K\to \mathbb{R}$ with bounded derivative. Then $X$ is a hyperplane and $u$ is the restriction of an affine function.
\end{corollary}

\begin{corollary}\label{cor:_weak_gen_bernstein_thm}
Let $X\subset \mathbb R^{n+1}$ be a connected complete minimal hypersurface with $2\leq n\leq 7$. Suppose that there are compact sets $K\subset \R^n$ and $\tilde K\subset \R^{n+1}$ such that $X\setminus \tilde K$ is the graph of a $C^2$-smooth function $u\colon \mathbb{R}^{n}\setminus K\to \mathbb{R}$. Then $X$ is a hyperplane and $u$ is the restriction of an affine function.
\end{corollary}

We finish this section by presenting the following example, which shows that the hypothesis that $X$ is a complete hypersurface in the above corollaries cannot be dropped:
\begin{example}
Outside of a compact set, the upper part of the catenoid is the graph of a Lipschitz function (see Example \ref{example:catenoid}). Indeed, we can see the upper part of the catenoid as the graph of the function $f\colon \R^2\setminus \overline{B_2^2}\rightarrow \R$ given by $f(x,y)=\mbox{arccosh}(\sqrt{x^2+y^2})$. Now, we note that 
\begin{eqnarray*}
\left|\frac{\partial f}{\partial x}(x,y)\right| & = & \frac{x}{\sqrt{x^2+y^2}}\frac{1}{\sqrt{x^2+y^2-1}} \leq 1,
\end{eqnarray*}
and
\begin{eqnarray*}
	\left|\frac{\partial f}{\partial y}(x,y)\right| & = & \frac{y}{\sqrt{x^2+y^2}}\frac{1}{\sqrt{x^2+y^2-1}} \leq 1.
\end{eqnarray*}
Since $\R^2\setminus \overline{B_2^2}$ is Lipschitz normally embedded set, we have that $f$ is Lipschitz (see Claim \ref{claim:Lipschitz_function}).
\end{example}

\bigskip

\noindent{\bf Acknowledgements}. The authors would like to thank Alexandre Fernandes and Pacelli Bessa for their interest, encouragement, and comments on this research.

\end{document}